\documentclass[a4paper,11pt]{article}
\usepackage{amsmath,amsthm}
\usepackage{amscd}
\usepackage{amssymb}
\usepackage[english]{babel}
\usepackage{verbatim}
\usepackage[shortlabels]{enumitem}
\usepackage{color}
\usepackage{tikz-cd}
\usepackage{mathrsfs}
\input xy
\xyoption{all}

\usepackage{vmargin}
\setmarginsrb{9mm}{0mm}{10mm}{9mm}%
          {10mm}{7mm}{10mm}{12mm}
\usepackage{xspace}

\newtheorem{lemma}{Lemma}[section]
\newtheorem{proposition}[lemma]{Proposition}
\newtheorem{theorem}[lemma]{Theorem}
\newtheorem{corollary}[lemma]{Corollary}

\newtheorem{question}[lemma]{Question}

\theoremstyle{definition}
\newtheorem{definition}[lemma]{Definition}
\newtheorem{remark}[lemma]{Remark}
\newtheorem{example}[lemma]{Example}

\def\N{\mathbb N}

\def\R{\mathbb R}
\def\Z{\mathbb Z}

\def\P{\mathcal P}
\def\Pf{\P_{fin}}

\def\supp{\mathrm{supp}}

\definecolor{agbgreen}{rgb}{0.1, 0.4, 0.0}

\newcommand{\T}{\mathcal{T}}

\newlength{\bibitemsep}
\newlength{\bibparskip}
\let\oldthebibliography\thebibliography
\renewcommand\thebibliography[1]{%
  \oldthebibliography{#1}%
  \setlength{\parskip}{\bibitemsep}%
  \setlength{\itemsep}{\bibparskip}%
}

\numberwithin{equation}{section}

\begin{document}

\title{The algebraic entropy of one-dimensional finitary linear cellular automata}

\date{}

\author{H. Ak\i n\footnote{The Abdus Salam International Centre for Theoretical Physics, Strada Costiera 11 I - 34151 Trieste, Italy; akin@ictp.it 
\endgraph \hfill
\ \ \ \ \ \ \ \ \ \ \ \ \ \ \ \ \ \ \ \ 
 (and 
Department of Mathematics, Faculty of Sciences and Arts, Harran University, TR63100, \c Sanl\i urfa, Turkey)} \and 
D. Dikranjan \footnote{University of Udine, via delle Scienze 206, Udine, Italy; dikran.dikranjan@uniud.it} \footnote{The second and the third author are members of the group GNSAGA of INdAM}  \and
A. Giordano Bruno\footnote{University of Udine, via delle Scienze 206, Udine, Italy; anna.giordanobruno@uniud.it} ${}^\ddag$ \and 
D. Toller\footnote{University of Aalborg, Selma Lagerl\o fs Vej 300, 9220 Aalborg, Denmark; danieleto@cs.aau.dk}}

\maketitle

\abstract{The aim of this paper is to present one-dimensional finitary linear cellular automata $S$ on $\Z_m$ from an algebraic point of view. 
Among various other results, we: \\
(i) show that the Pontryagin dual $\widehat S$ of $S$ is a classical one-dimensional linear cellular automaton $T$ on $\Z_m$;\\
 (ii) give several equivalent conditions for $S$ to be invertible with inverse a finitary linear cellular automaton; \\
(iii)  compute the algebraic entropy of $S$, which coincides with the topological entropy of $T=\widehat S$ by the so-called Bridge Theorem.\\
 In order to better understand and describe the entropy we introduce the degree $\deg(S)$ and $\deg(T)$ of $S$ and $T$.}
 
\bigskip
\noindent\hrulefill


\noindent\hrulefill


\section{Introduction}

The concept of cellular automaton was first introduced by Ulam and von Neumann~\cite{VonN} in the 1940s. Later, Hedlund~\cite{Hed} analyzed in detail cellular automata from a mathematical perspective. Since then, cellular automata were extensively studied by scientists in many fields, and nowadays they are one of the most compelling instruments for studying complicated systems in addition to having inner interest. They have been frequently used to model economic, political, biological, ecological, chemical, and physical systems due to their mathematical simplicity.

\smallskip
Following~\cite{CC}, let $A$ be a non-empty set (called \emph{alphabet}), $G$ a group and let $A^G=\{x\colon G\to A : x\ \text{map}\}$ be the \emph{set of configurations}.  For $g\in G$, denote by $L_g\colon G\to G$ the left multiplication by $g$.  
\begin{definition}
Let $A$ be a non-empty set and $G$ a group.
A map $T\colon A^G\to A^G$ is a \emph{cellular automaton} if there exist a finite non-empty subset $M$ of $G$ and a map $f\colon A^M\to A$ such that, for every $x\in A^G$ and $g\in G$, 
$$T(x)(g):=f((x\circ L_g)\restriction_M).$$
The set $M$ is called \emph{memory set} of $T$ and $f$ is the \emph{local rule} of $T$ associated to $M$. 
\end{definition}

For every memory set $M$ of $T$, the local rule associated to $M$ is necessarily unique, while there may exist different memory sets of $T$; nevertheless, by intersecting them, one can obtain a unique memory set of minimal cardinality, called \emph{minimal memory set}.

For a group $G$ and a non-empty set $A$, the surjectivity of a cellular automaton $T\colon A^G \to A^G$ has been thoroughly studied. A configuration $c \in A^G$ is called a \emph{Garden of Eden configuration for $T$} if $c$ is not in the image of $T$. Thus, the surjectivity of $T$ is equivalent to the non-existence of Garden of Eden configurations.  
We formulate the celebrated Garden of Eden Theorem, describing the surjectivity of a cellular automaton $T\colon A^G \to A^G$. (Recall that $T$ is said to be \emph{pre-injective} if any two configurations $c,c' \in A^G$ satisfying $|\{g \in G : c(g) \neq c'(g) \}| < \infty$ and $T(c) = T(c')$ coincide; obviously, injective implies pre-injective.)

\begin{theorem}{\rm\cite[Theorem 5.3.1]{CC}}\label{thm:goe}
Let $G$ be an amenable group, $A$ be a non-empty finite set and $T\colon A^G \to A^G$ a cellular automaton. Then $T$ is surjective if and only if it is pre-injective.
\end{theorem}

Moreover, a group $G$ is called \emph{surjunctive} if for every non-empty finite set $A$, every injective cellular automaton $T\colon A^G\to A^G$ is surjective. The class of surjunctive groups is very large, as it is known that every sofic group is surjunctive. For example, Theorem~\ref{thm:goe} yields that amenable groups are surjunctive. In particular, $\Z$ is surjunctive, and we discuss this case in \S\ref{sec:InjSurj}.

\smallskip
Also the following result holds in the case when $A$ is finite (a counterexample in the general case can be found in~\cite{CC}). 
Our choice to focus this paper on the case of a finite alphabet $A$ was determined by this relevant  difference. As usual, $A$ is endowed with the discrete (and compact) topology and $A^G$ with the product topology, so that $A^G$ is a Hausdorff, compact and totally disconnected topological space. Consider the left action of $G$ on $A^G$, called \emph{$G$-shift}, defined, for every $g\in G$ and every $x\in A^G$, by $$gx:=x\circ L_{g^{-1}}\in A^G.$$

\begin{theorem}[Curtis-Hedlund]\label{CHTheorem}
Let $G$ be a group, $A$ a non-empty finite set and $T\colon A^G\to A^G$ a map. Then $T$ is a cellular automaton if and only if $T$ is continuous and commutes with the $G$-shift. 
\end{theorem}

We shall ask the finite alphabet $A$ to be the finite cyclic group $\Z_m=\Z/m\Z$ with $m>0$, so that $\Z_m^G$ becomes a totally disconnected compact abelian group. 
Furthermore, we restrict to the classical case $G = \Z$, thus considering the so-called \emph{one-dimensional} cellular automata. 
\begin{center}
In the sequel we often omit ``one-dimensional", especially when it is clear from the context.
\end{center}

For a map $T\colon \Z_m^\Z\to \Z_m^\Z$, Theorem~\ref{CHTheorem} yields that $T$ is a cellular automaton if and only if $T$ is continuous and $T\circ \sigma=\sigma\circ T$, where $\sigma$ is the $\Z$-shift ($\sigma$ is usually called right shift, see Example~\ref{Exa1}).
A cellular automaton $T$ is called \emph{linear} if $T$ is a group endomorphism (equivalently, when its local rule $f$ is a homomorphism); moreover, this is equivalent to have $f$ in the form \eqref{locRule} of Definition~\ref{DefCA}. For every $\lambda\in\Z_m^\Z$, as usual we let $\supp(\lambda):=\{i\in\Z:\lambda_i\neq0\}$.

\begin{definition}\label{DefCA}
A map $T=T_{f[l,r]\lambda}\colon \Z_m^\Z\to \Z_m^\Z$ is a \emph{one-dimensional linear cellular automaton on $\Z_m$} if there exist a finite subset $M=\{l,\ldots,r\}$ of $\Z$ with $l\leq r$, $\lambda=(\lambda_n)_{n\in\Z}\in\Z_m^\Z$ with $\supp(\lambda)\subseteq M$ and a map $f\colon \Z_m^M\to \Z_m$, such that, for every $x=(x_n)_{n\in\Z}\in \Z_m^\Z$,
\begin{equation}\label{lCA}
(T(x))_n=f(x_{n+l},\ldots,x_{n+r})\ \text{for every $n\in\Z$},
\end{equation}
where, for every $(x_l,\ldots,x_r)\in \Z_m^M$,
\begin{equation}\label{locRule}
f(x_l,\ldots,x_r)=\sum_{i=l}^r\lambda_i x_i.
\end{equation}
\end{definition}

\begin{remark}\label{remark:june:8}
\begin{itemize}
  \item[(1)] In the above definition, $\supp(\lambda)$ is the minimal memory set for $T_{f[l,r]\lambda}$. Obviously, \eqref{locRule} implies that 
$f(x_l,\ldots,x_r)=\sum_{ i\in \supp(\lambda) }\lambda_i x_i$. In particular, $f$ is determined by $\lambda$ (via its entries, and its support), and one could denote simply 
$$T_{f[l,r]\lambda}=T_{\lambda}$$ when there is no need to specify the local rule and the memory set. 
  \item[(2)] When $f\equiv0$ (i.e., $\lambda=0$), then also $T\equiv0$, and vice versa; in this case $\supp(\lambda)=\emptyset$. When $T\not\equiv0$, one can always assume without loss of generality that $l,r\in\supp(\lambda)$.
  \item[(3)] In the literature (e.g., in \cite{DMM}) the set $M$ in the above definition is often taken to be symmetric, that is, $M=\{-r,\ldots,r\}$ for some $r\in\N$, and at least one
between $r$ and $-r$ is required to be in $\supp(\lambda)$. In this paper we prefer to use the above non-symmetric notation and the minimal memory set $\supp(\lambda)$.
\end{itemize}
\end{remark}

\begin{example}\label{Exa1}
\begin{itemize}
   \item[(1)]The identity automorphism $id_{\Z_m^\Z}\colon \Z_m^\Z\to \Z_m^\Z$ can be obtained as the linear cellular automaton $T_{f[l,r]\lambda}$ with $l=r=0$ and local rule $f(x_0)= x_0$, that is, $\supp(\lambda)=\{0\}$ and $\lambda_0=1$.
   \item[(2)]The {\em right Bernoulli shift} $\sigma_m\colon \Z_m^\Z\to \Z_m^\Z$ is given by $T_{f[l,r]\lambda}$ with $l=r=-1$ and local rule $f(x_{-1})= x_{-1}$, that is, $\supp(\lambda)=\{-1\}$ and $\lambda_{-1}=1$. 
   \item[(3)] The {\em left Bernoulli shift} $\beta_m\colon \Z_m^\Z\to \Z_m^\Z$ is given by $T_{f[l,r]\lambda}$ with $l=r=1$ and local rule $f(x_1)= x_{1}$, that is, $\supp(\lambda)=\{1\}$ and $\lambda_1=1$; clearly, $\beta_m$ is the inverse of $\sigma_m$.
\end{itemize}
The three linear cellular automata above have the same local rule (the identity map) and $|\supp(\lambda)| = |M| = 1$ in all three cases, but they differ in the choice of the memory set $M$ (and so of $\lambda$). When there is no possibility of confusion, we denote $\sigma_m$ and $\beta_m$ simply by $\sigma$ and $\beta$, respectively.
\end{example}

\smallskip
Many papers were dedicated to dynamical aspects of one-dimensional linear cellular automata on $\Z_m$, as for example~\cite{Akin1,Akin2,BKM,DMM0,DMM,DiLena}.
This paper is focused on the entropy of one-dimensional linear cellular automata on $\Z_m$. As far as the topological entropy is concerned, since the compact abelian group $\Z_m^\Z$ is metrizable, the definition of topological entropy by Adler, Konheim and McAndrew~\cite{AKM} by open covers of compact spaces coincides with that given by Bowen~\cite{B} for uniformly continuous self-maps of metric spaces.
Hurd, Kari and Culik~\cite{HKC} showed that the topological entropy of one-dimensional cellular automata cannot be computed algorithmically in the general case. 
The relevance of {linearity} became clear when D'Amico, Manzini and Margara~\cite[Theorem 2]{DMM} showed that the topological entropy of one-dimensional linear cellular automata on $\Z_m$ can be computed by a precise formula (see Corollary~\ref{htopFClCA}).

\medskip
 Following~\cite{CC,ION}, call the direct sum $\Z_m^{(\Z)}=\{c\in \Z_m^\Z: |\supp(c)|<\infty\}\leq \Z_m^\Z$ the \emph{set of finite configurations}.
Since $\Z_m^{(\Z)}$ remains invariant under the action of any linear cellular automaton $T\colon \Z_m^\Z\to \Z_m^\Z$, this terminology also motivates the following definition. 

\begin{definition}\label{Def:June6}\cite{ION} 
If $T_{f[l,r]\lambda}\colon\Z_m^\Z\to \Z_m^\Z$ is a linear cellular automaton, its restriction 
$$S_{f[l,r]\lambda}:=T_{f[l,r]\lambda}\restriction_{\Z_m^{(\Z)}}\colon \Z_m^{(\Z)}\to \Z_m^{(\Z)}$$
 is called \emph{one-dimensional finitary linear cellular automaton on $\Z_m$} with memory set $M=\{l,\ldots,r\}$ and local rule $f\colon\Z_m^M\to\Z_m$ associated to $M$. 
\end{definition}

In the sequel the domain $\Z_m^{(\Z)}$ of a finitary linear cellular automaton $S\colon\Z_m^{(\Z)}\to\Z_m^{(\Z)}$ will always be considered as a {\em discrete} abelian group, unless otherwise explicitly specified.

Following Remark~\ref{remark:june:8}(1), we simply write $S_\lambda$ instead of $S_{f[l,r]\lambda}$ when we do not need explicitly $l,r$ and $f$.

\begin{example}
We still denote by $\sigma_m\colon\Z_m^{(\Z)}\to\Z_m^{(\Z)}$ (respectively, by $\beta_m\colon\Z_m^{(\Z)}\to\Z_m^{(\Z)}$) the restriction of the right (respectively, left) shift to the set of finite configuration, and write it simply $\sigma$ (respectively, $\beta$) when there is no possibility of confusion.
\end{example}

 The properties of one-dimensional finitary linear cellular automata on $\Z_m$ and their applications are the main object of study of this paper.

\smallskip
For a locally compact abelian group $L$, we denote by $\widehat L$ the Pontryagin dual of $L$, which is again a locally compact abelian group; if $\phi\colon L\to L$ is a continuous endomorphism of $L$, we denote by $\widehat\phi\colon \widehat L\to\widehat L$ its dual endomorphism, which is still continuous. See \S\ref{pd-sec} for more details on Pontryagin duality. In particular, we need that $\widehat{\Z_m}\cong\Z_m$, and so we identify $\widehat{\Z_m}$ with $\Z_m$ and $\widehat{\Z_m^{(\Z)}}$ with $\Z_m^\Z$.

Here comes one of the main results of the paper,  describing the Pontryagin dual of a one-dimensional finitary linear cellular automaton $S_{\lambda}\colon \Z_m^{(\Z)} \to  \Z_m^{(\Z)}$. For $\lambda = ( \lambda_n )_{n \in \Z} \in \Z_m^{\Z}$ we use the notation 
$$
\lambda^\wedge:=( \lambda_{-n} )_{n\in\Z} \in \Z_m^{\Z}
$$
and we call $\lambda$ \emph{symmetric} if $\lambda^\wedge=\lambda$. 

\begin{theorem}\label{dualCAintro}
If $S_{\lambda}\colon\Z_m^{(\Z)}\to \Z_m^{(\Z)}$ is a finitary linear cellular automaton, then $\widehat {S_{\lambda}}=T_{\lambda^\wedge}$. Consequently, if $T_{\lambda}\colon \Z_m^\Z\to \Z_m^\Z$ is a linear cellular automaton, then $\widehat {T_{\lambda}}=S_{\lambda^\wedge}$. In case $\lambda$ is symmetric, $\widehat {S_{\lambda}}=T_{\lambda}$ and $\widehat {T_{\lambda}}=S_{\lambda}$.
\end{theorem}

As an easy consequence of this result, we find a new simpler proof of the known relation from~\cite{ION,Rich} that $T_{\lambda}$ is injective (respectively, surjective) if and only if $S_{\lambda}$ is surjective (respectively, injective) (see Corollary~\ref{injsurjsurjinj}).

\medskip
Using ideas briefly sketched by Adler, Konheim and McAndrew~\cite{AKM} in connection with the topological entropy, Weiss~\cite{W} developed the definition of algebraic entropy for endomorphisms of torsion abelian groups; this notion of entropy was rediscovered in~\cite{DGSZ}, where it was studied in deep. Moreover, Peters~\cite{P1} modified Weiss' notion of algebraic entropy for automorphisms of abelian groups, and his approach was extended to all endomorphisms of abelian groups in~\cite{DGB2}. For the algebraic entropy of continuous endomorphisms of locally compact abelian groups, see~\cite{P2,V1}. 

The algebraic entropy is connected to the topological entropy by means of Pontryagin duality in the so-called Bridge Theorem (see Theorem~\ref{BT}), namely, the algebraic entropy of an endomorphism of a discrete abelian group coincides with the topological entropy of its dual. This was proved in special cases in~\cite{P1,W} and in the general case in~\cite{DGB1}. For the case of locally compact abelian groups and their continuous endomorphisms, see~\cite{DGB2,P2}. See also~\cite{DGB4} for further connections among algebraic, topological and measure entropies.

\smallskip
Our main result in this paper is the computation of the algebraic entropy of one-dimensional finitary linear cellular automata on $\Z_m$ in Theorem~\ref{halgFClCA}.

In the sequel, for $\lambda\in\Z_m^\Z$, we frequently make use of the {\em reduced support}
$$
\supp^*(\lambda) :=\{i\in \Z: (\lambda_i, m) =1\}
$$
of $\supp(\lambda)$.  We denote by $(\lambda_i ,m)$ the greatest common divisor of $m$ and any representative in $\Z$ of the class $\lambda_i\in \Z_m$. It is useful to define also 
 $$
 \supp^* (S_{\lambda}):= \supp^* (\lambda)\quad \text{and}\quad  \supp^*(T_{\lambda}):= \supp^* (\lambda).
 $$

\begin{definition}\label{degree}
Let $S=S_{\lambda}\colon \Z_m^{(\Z)}\to \Z_m^{(\Z)}$ be a finitary linear cellular automaton. Define the {\em degree} $\deg(S)$ of $S$ as follows. If $\supp^*(\lambda)=\emptyset$ (in particular, if $S\equiv 0$), let $\deg(S)=0$; otherwise, if  $l':=\min \supp^*(\lambda)$ and $r':=\max \supp^*(\lambda)$, let  
$$\deg(S):=\begin{cases}-l' & \text{if $l'\leq r'\leq 0$},\\ r' & \text{if $0\leq l'\leq r'$},\\ r'-l' & \text{if $l'<0<r'$}.\end{cases}$$
The same definition of degree $\deg(T)$ can be given for a linear cellular automaton $T\colon \Z_m^{\Z}\to \Z_m^{\Z}$.
\end{definition}

 In the case when $m$ is a prime power we obtain the following simple formula.

\begin{theorem}\label{Th5} 
Let $m=p^k$ with $p$ prime and $k\in\N_+$, and let $S\colon \Z_m^{(\Z)}\to \Z_m^{(\Z)}$ be a finitary linear cellular automaton.  Then $$h_{alg}(S)=\deg(S)\log p^k.$$
\end{theorem}

 The above conclusion $h_{alg}(S)=\deg(S)\log m$ does not hold in the general case when $m$ is not a prime power (see Example~\ref{non-injex}). In fact, we have the following general result.

\begin{theorem} \label{halgFClCA}\label{finitaryteo}
Let $S\colon \Z_m^{(\Z)}\to \Z_m^{(\Z)}$ be a finitary linear cellular automaton.
Let $m=p_1^{k_1}\cdots p_h^{k_h}$ be the prime factorization of $m$, and for every $i\in\{1,\ldots,h\}$, let $L_i:= \Z_{p_i^{k_i}}^{(\Z)}$ denote the direct factor of  $\Z_m^{(\Z)}\cong L_1\times\ldots\times L_h$. 
 Then each $S^{(i)}:=S\restriction_{L_i}\colon L_i\to L_i$  is a finitary linear cellular automaton and 
$$h_{alg}(S)=\sum_{i=1}^h \deg(S^{(i)})\log p_i^{k_i}.$$
\end{theorem}

As a consequence of Theorems~\ref{dualCAintro} and~\ref{finitaryteo} and the Bridge Theorem~\ref{BT}, we get exactly the value of the topological entropy of a linear cellular automaton $T\colon \Z_m^{\Z}\to \Z_m^{\Z}$ as computed in~\cite[Theorem 2]{DMM} (see Corollary~\ref{htopFClCA}). Other applications will be given below.

\medskip
The paper is organized as follows.

The initial \S\ref{prelsec} contains preliminary results and discussions. In particular, in \S\ref{pd-sec} we briefly recall what we need about Pontryagin duality. 

In \S\ref{Background} we recall the formal power series representation of configurations $c\in\Z_m^\Z$ observing that the finite formal power series are exactly the elements of the ring $R_m:=\Z_m[X,X^{-1}]$ of the Laurent polynomials on $\Z_m$. For any $A(X)\in R_m$, we introduce its reduced Laurent polynomial $A^*(X)\in R_m$, which allows us to 
define the notion of reduced degree $\deg^*(A(X))$ of $A(X)$.

In \S\ref{Background*} first we prove the counterpart of Theorem~\ref{CHTheorem} for one-dimensional finitary linear cellular automata on $\Z_m$. Theorem~\ref{CHTheorem} (and its counterpart, respectively) allows us to see the family $\mathrm{CA}(\Z_m)$ (respectively, $\mathrm{CA}_f(\Z_m)$) of all one-dimensional (finitary) linear cellular automata on $\Z_m$ as a $\Z_m$-algebra generated by the shifts $\sigma_m$ and $\beta_m=\sigma_m^{-1}$. Then, we recall the formal power series representation $A_T(X)\in R_m$ (respectively, $A_S(X)\in R_m$) of a one-dimensional (finitary) linear cellular automaton $T$ (respectively, $S$) on $\Z_m$, and this association gives a $\Z_m$-algebra isomorphism between $\mathrm{CA}(\Z_m)$ (respectively, $\mathrm{CA}_f(\Z_m)$) and $R_m$ such that $\deg(T)=\deg^*(A_T(X))$ (respectively, $\deg(S)=\deg^*(A_S(X))$).

Finally, \S\ref{primaryfactorization} is dedicated to the primary factorization of a one-dimensional (finitary) linear cellular automaton, which is fundamental for the computation of its topological (respectively, algebraic) entropy.

\smallskip
In \S\ref{like:a:scalar} we prove Theorem~\ref{dualCAintro}. 
As an immediate corollary, by exploiting the properties of Pontryagin duality, we conclude that the linear cellular automaton $T_{\lambda}\colon\Z_m^\Z\to\Z_m^\Z$ is injective (respectively, surjective) if and only if the finitary linear cellular automaton $S_{\lambda}\colon\Z_m^{(\Z)}\to \Z_m^{(\Z)}$ is surjective (respectively, injective).

\smallskip
In \S\ref{sec:InjSurj} we prove that surjective one-dimensional finitary linear cellular automata on $\Z_m$ are necessarily invertible with inverse a one-dimensional finitary linear cellular automaton. As a consequence of this result and Theorem~\ref{dualCAintro}, we get an alternative proof of the known fact that every injective one-dimensional linear cellular automaton on $\Z_m$ is surjective (compare this with the above mentioned surjunctivity of $\Z$).

\smallskip 
In \S\ref{halgS} we compute the algebraic entropy of a one-dimensional finitary linear cellular automaton $S$ on $\Z_m$. In several steps we arrive at the proof of Theorem~\ref{finitaryteo}: in \S\ref{SSec1} we prove it in case $S$ is permutive, in \S\ref{Sec2} we treat the case when $m$ is a power of a prime and reduce to the permutive case by using the properties of the formal power series representation, and then we settle the general case by exploiting the properties of the algebraic entropy, recalled in \S\ref{Sec:Feb26}.

 The last \S\ref{finalsec} contains some final remarks and open problems. First, in \S\ref{astheshift} we see that when $m=p^k$ for a prime $p$ and $k\in\N_+$,
the shifts are ``dominant" to the effect of the values of the algebraic entropy, in the sense that a one-dimensional finitary linear cellular automaton $S$ on $\Z_{p^k}$ has the same algebraic entropy as  $\sigma_{p^k}^{\deg(S)}$. 
This cannot be generalized to the case of an arbitrary $m$ and we characterize those one-dimensional finitary linear cellular automata on $\Z_m$ that have the same algebraic entropy of an integer power of the right shift $\sigma_m$.
 Finally, in \S\ref{lastsubsub}, we consider the stronger property for a one-dimensional finitary linear cellular automaton $S$ on $\Z_{m}$ to be conjugated to $\sigma_m^n$ for some $n\in\Z$. 
%
%

\subsection*{Notation and terminology}

We denote by $\N$ the set of natural numbers, by $\Z$ the set of integers and by $\N_+$ the set of strictly positive integers.

For $i \in \Z$, we denote by $e_i \in L$ the element whose $n$-th coordinate is, for every $n \in \Z$,
\[(e_i)_n =\begin{cases}
1 & \text{if } n =i,\\
0 & \text{otherwise}.
\end{cases}\]
We denote by $(\Z_m)_i = \{ x \in L : \supp(x) \subseteq \{i\}\}$ the subgroup of $L$ generated by $e_i$. So, $\Z_m^\Z = \prod_{i\in \Z}(\Z_m)_i$ and $\Z_m^{(\Z)}=\bigoplus_{i\in\Z}(\Z_m)_i$.

\section{Preliminaries}\label{prelsec}

\subsection{Background on Pontryagin duality}\label{pd-sec}

We recall here the definitions and properties related to Pontryagin duality needed in this paper, following~\cite{ADGB}.

For a locally compact abelian group $L$, the Pontryagin dual $\widehat L$ of $L$ is the group of all continuous homomorphisms (i.e., characters) $\chi\colon L\to \mathbb T:=\R/\Z$, endowed with the compact-open topology; $\widehat L$ is a locally compact abelian group as well.
Moreover, for a continuous homomorphism $\phi\colon L\to M$ of locally compact abelian groups, the dual of $\phi$ is the continuous homomorphism $\widehat\phi\colon\widehat M\to \widehat L$, such that $\widehat\phi(\chi)=\chi\circ\phi$ for every $\chi\in\widehat M$. In particular, denoting by $\mathcal L$ the category of all locally compact abelian groups and their continuous homomorphisms, $\widehat{\ }\colon\mathcal L\to \mathcal L$ is a contravariant functor.

 The Pontryagin duality Theorem asserts that every locally compact abelian group $L$ is canonically isomorphic to its bidual $\widehat{\widehat L}$, by means of the canonical topological isomorphism $\omega_L\colon L\to\widehat{\widehat L}$ such that, for every $x\in L $ and $\chi\in\widehat L$, $\omega_L(x)(\chi)=\chi(x)$. 
More precisely, $\omega$ is a natural equivalence from the identity functor $1_\mathcal L$ to the functor $\widehat{\widehat{\ }}\colon \mathcal L\to \mathcal L$.

In this paper, $L$ will always be either discrete or compact. Recall that $L$ is finite if and only if $\widehat L$ is finite, and in this case $\widehat L\cong L$. 
 On the other hand, $L$ is compact precisely when $\widehat L$ is discrete, and under this assumption $L$ is totally disconnected if and only if $\widehat L$ is torsion.  The Potryagin dual group of a direct product $\prod_{i\in I}L_i$ of compact abelian groups is (isomorphic to) the direct sum $\bigoplus_{i\in I}\widehat L_i$ of the dual groups of the $L_i$, while the Potryagin dual group of a direct sum $\bigoplus_{i\in I} D_i$ of discrete abelian groups is  (isomorphic to) the direct product $\prod_{i\in I}\widehat D_i$.
 
In this paper we are mainly interested in the following very special case.

\begin{example}
The dual of the totally disconnected compact abelian group $\Z_m^\Z$ is, up to isomorphism, the discrete torsion abelian group $\Z_m^{(\Z)}$.
\end{example}

\subsection{Formal power series}\label{Background}

The formal power series (briefly, fps) provide a practical notation to deal with one-dimensional linear cellular automata (see~\cite[Section 3]{ION} and \S\ref{Background*}). More precisely, each configuration $c =( c_i)_{i\in\Z}\in \Z_m^\Z$ can be encoded as the fps $P_{c}(X)$ defined by
\[P_{c}(X):=\sum_{i\in \Z} c_i X^{i}.\]

\begin{remark}\label{easy:remark}
Clearly, $P_{c^\wedge}(X)= P_c (X^{-1})$ for every $c\in\Z_m^\Z$.
\end{remark}

For $c\in\Z_m^\Z$, we are going to use the \emph{support} $\supp(P_c(X)):= \supp(c)$ and the {\em reduced support} $\supp^*(P_c(X)):= \supp^*(c)$, respectively, of $P_c(X)$,  as already defined in the introduction. We briefly say that a fps $P_c(X)$ is \emph{finite} if it has a finite support (obviously this occurs exactly when $c \in \Z_m^{(\Z)}$ is a finite configuration). 

\smallskip 
The fps used to represent configurations in $\Z_m^{\Z}$ have no limitation on the support, unlike the fps commonly used  in commutative algebra and elsewhere, where they are supposed to have bounded from below supports which enables one to multiply them. In particular, a finite fps $A(X)$ is actually a Laurent polynomial, i.e., $A(X) \in R_m:= \Z_m[X,X^{-1}]$, where $R_m$ is the ring of Laurent polynomials with coefficients in $\Z_m$. 

\smallskip

The next definition is motivated by the dynamical properties of Laurent polynomials.

\begin{definition}\label{Def:A*}
 For $A(X)\in R_m$, the {\em reduced Laurent polynomial} of $A(X)$ is 
\[A^*(X):= \sum_{i\in \supp^*(A(X))} \lambda_i X^{i},\] 
where $A^*(X):= 0$ if $\supp^*(A(X))= \emptyset$. 
\end{definition}

Here comes a first application of the reduced Laurent polynomial.

\begin{lemma}\label{aboveabove} 
Let $m=p^k$ with $p$ prime and $k\in\N_+$, and let $A(X)\in R_m$. Then 
$
A (X)^{p^{k-1}} = A^*(X)^{p^{k-1}}.
$
\end{lemma}
\begin{proof} 
There exists $A_1(X)\in R_m$ such that 
$$A(X)=A^*(X) + pA_1(X).$$
Then $A(X)^p = A^*(X)^p + p^2A_2(X)$ for an appropriate $A_2(X)\in R_m$, and similarly $A(X)^{p^2} = A^*(X)^{p^2} + p^3A_3(X)$ for an appropriate $A_3(X)\in R_m$, etc. By induction this gives the required equality. 
\end{proof}

Fixed $A(X)\in R_m$, one can write $A^*(X)$ as a sum $$A^*(X)= A_-(X^{-1}) + A_+(X)$$ of two polynomials $A_-(X^{-1})\in\Z_m[X^{-1}]$ and $A_+(X)\in\Z_m[X]$, both with coefficients coprime to $m$. Then the degrees of the polynomials $\deg(A_+(X))$ and $\deg (A_-(X))$ are well defined in case they are non-zero; for the sake of technical convenience, we set $\deg(A_+(X))=0$ and  $\deg (A_-(X))=0$, when these polynomials are zero. 

\begin{definition}\label{Def/Rem:grado} For $A(X)\in R_m$, the \emph{reduced degree} of $A(X)$ is
$$\deg^*(A(X)):= \deg(A^*(X)):=  \deg(A_+) + \deg (A_-).$$
\end{definition}

Fix a prime factorization $m=p_1^{k_1}\cdots p_h^{k_h}$ of $m$, and for every $i\in\{1,\ldots,h\}$ let $m_i=p_i^{k_i}$.
Clearly, there is a ring isomorphism $\iota\colon\Z_m\to \Z_{m_1}\times\ldots\times \Z_{m_h}$, defined by $1\mapsto (1,\ldots,1)$, and for every $i\in\{1,\ldots,h\}$, let $\pi_i\colon \Z_m\to \Z_{m_i}$ be the canonical projection, that is, the composition of $\iota$ with the $i$-eth canonical projection $\Z_{m_1}\times\ldots\times \Z_{m_h}\to \Z_{m_i}$, so $\pi_i(a)$, for $a\in \Z_m$, is the remainder of $a$ modulo $m_i$. 
Since $R_m$ coincides with the subgroup $\Z_m^{(\Z)}$ of $\Z_m^\Z$, for every $i\in\{1,\ldots,h\}$,  the power of the projection $\pi_i^\Z\colon \Z_m^{\Z}\to \Z_{m_i}^\Z$ 
takes $R_m$ to $R_{m_i}$, hence  the restriction $\varpi_i=\pi_i^\Z\restriction_{R_m}\colon R_m \to R_{m_i}$ is a surjective ring homomorphism. For $\lambda \in\Z_m^{(\Z)}=R_m$ we write $\bar \lambda_i:= \varpi_i(\lambda)$, so that for $A(X)=\sum_{i=l}^r\lambda_i X^i \in R_m$,
\begin{equation}\label{Ai}
A_i(X):= \varpi_i (A(X))=\sum_{i=l}^r\bar\lambda_i X^i \in R_{m_i}.
\end{equation}
Putting $ P:= R_{m_1}\times\ldots\times R_{m_h}$ and $$j(A(X))=(A_1(X),\ldots, A_h(X))$$
we obtain a ring homomorphism 
\begin{equation}\label{jeq}
j\colon R_m\to P
\end{equation}
such that  the $i$-eth canonical projection $\tilde\pi_i\colon P \to R_{m_i}$ satisfies $\varpi_i= \tilde\pi_i \circ j$.  
Since $\bigcap_{i=1}^h \ker \varpi_i = \{0\}$ (as $\ker \varpi_i= (m/m_i) R_m$ for every $i\in\{1,\ldots,h\}$), $j\colon R_m\to P$ is a monomorphism, and actually, a ring isomorphism.

Hence, $j$ is also a group isomorphism with respect to the additive structure of the rings $R_m$ and $P$. From this point of view, the underlying abelian group of $P$ is also a direct sum of the groups $R_{m_i}$, with respect to the canonical group embeddings $R_{m_i}\to P$ defined as follows: 
\begin{equation}\label{barAi}
A_i(X) \mapsto \widetilde A_i(X) :=  (0,\ldots,0,A_i(X),0,\ldots,0) \ \mbox{ for } A_i(X)\in R_{m_i}.
\end{equation}
They provide a group isomorphism $\bigoplus_{i=1}^h R_{m_i} \to P$ which allows us to recover $A(X) \in R_m$ from its projections
$A_i(X)$ as follows: 
$$j(A(X))= \widetilde A_1(X)+\ldots + \widetilde A_h(X).$$

As we will see below, when dealing with linear cellular automata, we are interested in their composition (which corresponds to the product
in $R_m$), more than their sum, so now we would like to obtain $j(A(X))$ also as a product $B_1(X)\cdot\ldots\cdot B_h(X)$ of appropriately defined factors  $B_i(X)\in P$. In other words, our next aim is to obtain the {\em multiplicative monoid}  $(R_m,\cdot,1_m)$ also as a {direct sum} of the monoids $(R_{m_i} ,\cdot,1_{m_i})$. To this end we define the canonical monoid embeddings $j_i\colon R_{m_i}\to P$ by 
\begin{equation}\label{Bi}
j_i \colon A_i(X) \mapsto B_i(X):=(1,\ldots,1,A_i(X),1,\ldots,1)\ \text{for}\ A_i(X)\in R_{m_i}.
\end{equation} 
They provide a monoid isomorphism $\bigoplus_{i=1}^h R_{m_i} \to P$ which allows us to recover $A(X) \in R_m$ from: 
$$j(A(X))= B_1(X)\cdot\ldots \cdot B_h(X). $$

\subsection{Formal power series and linear cellular automata}\label{Background*}

\begin{definition}
Let $\mathrm{CA}(\Z_m)$ (respectively, $\mathrm{CA}_f(\Z_m)$) be the family of all one-dimensional (finitary) linear cellular automata on $\Z_m$. 
\end{definition}

The aim of this subsection is to recall the use of fps in the study of  $\mathrm{CA}(\Z_m)$ and $\mathrm{CA}_f(\Z_m)$ by means of some new notions we propose here. 

The following natural counterpart of Theorem~\ref{CHTheorem} turns out to be quite useful.

\begin{theorem}\label{finitaryCH}
A group endomorphism $S\colon\Z_m^{(\Z)}\to\Z_m^{(\Z)}$ is a finitary linear cellular automaton if and only if $S$ is continuous with respect to the topology on $\Z_m^{(\Z)}$ induced by the product topology and $S\circ \sigma=\sigma\circ S$.
\end{theorem}
\begin{proof} If $S=S_{f[l,r]\lambda}\in\mathrm{CA}_f(\Z_m)$, then $T=T_{f[l,r]\lambda}\in \mathrm{CA}(\Z_m)$, and so $T$ is continuous and $T\circ \sigma=\sigma\circ T$ by Theorem~\ref{CHTheorem}. As $S=T\restriction_{\Z_m^{(\Z)}}$, we conclude that $S$ is continuous and $S\circ \sigma=\sigma\circ S$.

Now assume that $S$ is continuous and $S\circ \sigma=\sigma\circ S$. As $\Z_m^\Z$ is the completion of $\Z_m^{(\Z)}$, there exists a unique continuous extension $T\colon \Z_m^\Z\to \Z_m^\Z$ of $S$ (see \cite[Corollary~7.1.20]{ADGB}). 
Analogously,  $\sigma^{-1}\circ S\circ \sigma$  has a unique continuous extension to $\Z_m^\Z$, that coincides with $\sigma^{-1}\circ T\circ\sigma$
since the continuous endomorphisms $T$ and $\sigma^{-1}\circ T\circ\sigma$ coincide on the dense subgroup $\Z_m^{(\Z)}$ 
with $S = \sigma^{-1}\circ S\circ \sigma$. 
Hence, $T\in \mathrm{CA}(\Z_m)$, by Theorem~\ref{CHTheorem}, and so $S\in \mathrm{CA}_f(\Z_m)$, 
being the restriction of $T$ to $\Z_m^{(\Z)}$.
\end{proof}

\begin{corollary} \label{CAalg}
$(\mathrm{CA}(\Z_m), + , \circ)$ and  $(\mathrm{CA}_f(\Z_m), + , \circ)$ are $\Z_m$-algebras. 
\end{corollary}
\begin{proof} 
Clearly, $(\mathrm{End}(\Z_m^\Z), + , \circ)$ is a $\Z_m$-algebra. As the operations $+$ and $\circ$ preserve the properties characterizing $\mathrm{CA}(\Z_m)$ within $\mathrm{End}(\Z_m^\Z)$ in Theorem~\ref{CHTheorem}, namely, continuity and commutation with $\sigma$, we get that $\mathrm{CA}(\Z_m)$ is a  $\Z_m$-algebra. 
To see that $(\mathrm{CA}_f(\Z_m), + , \circ)$ is a $\Z_m$-algebra apply a similar argument using Theorem~\ref{finitaryCH}. 
\end{proof}

\begin{definition}
Let $T=T_{f[l,r]\lambda}\colon \Z_m^{\Z}\to \Z_m^{\Z}$ be a linear cellular automaton. The \emph{finite fps associated to $T$} is
\begin{equation}\label{ffps}
A_T(X):= \sum_{i=l}^{r} \lambda_i X^{-i}\in R_m,
\end{equation} 
so $A_T(X) = 0$ when $T\equiv0$.
For $S=T\restriction_{\Z_m^{(\Z)}}$, the \emph{finite fps associated to $T$} is again $A_S(X):=A_T(X)$.
\end{definition}

\begin{remark}\label{suppAsuppT} 
Let $T=T_{f[l,r]\lambda} \in \mathrm{CA}(\Z_m)$.
\begin{itemize}
  \item[(1)] Definition \eqref{ffps} is inspired by the fact that $T(e_0) = \sum_{i=-r}^{-l} \lambda_{-i}e_i = \lambda^{\wedge}$ (see \eqref{Eqe_1} for a more general fact), so Remark~\ref{easy:remark} gives $ A_T(X) = P_{\lambda}(X^{-1}) = P_{ \lambda^{\wedge} }(X) = P_{T(e_0)}(X)$. 
  \item[(2)] One has $\supp(A_T(X)) = \supp(\lambda^{\wedge}) = -\supp(\lambda)$; analogously for $S=T\restriction_{\Z_m^{(\Z)}}$.
\end{itemize}
\end{remark}

Clearly, for any configuration $c\in \Z_m^\Z$, its image configuration $T(c)$ can be written as $P_{T(c)}(X)$.
The utility of the fps in the study of linear cellular automata is based on the following proposition, which generalizes the fact we pointed out in Remark~\ref{suppAsuppT}(1).

 \begin{theorem}\label{fpscor} Let $T\colon \Z_m^{\Z}\to \Z_m^{\Z}$ be a linear cellular automaton. 
\begin{itemize}
\item[(1)] For every $c\in\Z_m^\Z$,
\[P_{T(c)}(X) = A_T(X) P_c (X)=P_{T(e_0)}(X) P_c (X).\]
\item[(2)]  The maps
$$\rho\colon \mathrm{CA}(\Z_m)\to R_m,\ T\mapsto A_T(X),\quad\text{and}\quad \varrho\colon \mathrm{CA}_f(\Z_m)\to R_m,\ S\mapsto A_S(X),$$
are $\Z_m$-algebra isomorphisms preserving the degrees, i.e.,  
\begin{equation}\label{(3)}
\deg(T)= \deg^*(\rho(T))=\deg^*(A_T(X)) \ \mbox{ and } \ \deg(S)= \deg^*(\varrho(S))=\deg^*(A_S(X))
\end{equation}
for $T \in \mathrm{CA}(\Z_m)$ and $S \in \mathrm{CA}_f(\Z_m)$.
In particular, for $T, T_1,T_2 \in \mathrm{CA}(\Z_m)$ and   for every $n\in\N_+$, 
\begin{equation}\label{Laaast:Eq}
A_{T_1\circ T_2}(X)=A_{T_1}(X)A_{T_2}(X)\ \ \mbox{ and } \ \ A_{T^n}(X)=A_T(X)^n.
\end{equation}
\end{itemize}
\end{theorem}
\begin{proof} (1) 
Both assertions are obvious when $T\equiv0$. Assume now that $T\not \equiv 0$ and compute 
\begin{align*}
A_T(X) P_c (X) &= \left( \sum_{i=l}^{r}\lambda_i X^{-i} \right) \left( \sum_{j\in \Z} c_j X^{j} \right) = \sum_{j\in \Z} c_j \sum_{i=l} ^{r}\lambda_i X^{j-i} 
= \sum_{j\in \Z} \sum_{i=l}^{r} \lambda_i c_j X^{j-i} = \\
&= \sum_{n\in \Z} \sum_{i=l}^r\lambda_i c_{n+i} X^n= \sum_{n\in \Z} f(c_{n+l},\ldots,c_{n+r}) X^n = P_{T(c)}(X).
\end{align*}

To prove the first assertion in (2), it is enough to check that $\rho$ is a $\Z_m$-algebra isomorphism. 
Since $\rho$ is obviously $\Z_m$-linear, it remains to check that $A_{T_1\circ T_2}(X)=A_{T_1}(X)A_{T_2}(X)$. Indeed, 
 for every $c\in\Z_m^\Z$, 
$$A_{T_1\circ T_2}(X)P_c(X)=P_{T_1\circ T_2(c)}(X)=P_{T_1(T_2(c))}(X)=A_{T_1}(X)P_{T_2(c)}(C)=A_{T_1}(X)A_{T_2}(X)P_c(X),$$
and so $A_{T_1\circ T_2}(X)=A_{T_1}(X)A_{T_2}(X)$. The second equation in \eqref{Laaast:Eq} follows from the first one by induction. 
\end{proof}

For example, by Theorem~\ref{fpscor}(2) the sum in $\mathrm{CA}(\Z_m)$ is the linear cellular automaton $T_{\lambda} + T_{\lambda'} = T_{\lambda + \lambda'}$.

\begin{remark}\label{New:Remark} 
\begin{enumerate}[(1)]
\item Theorem~\ref{fpscor}(1) says, in the language of module theory, that both $\Z_m^{\Z}$ and $\Z_m^{(\Z)}$ are $R_m$-modules and the module multiplication by $A(X)\in R_m$ has the same effect as the action of the (finitary) linear cellular automaton $T\colon\Z_m^{\Z}\to\Z_m^\Z$ with $A(X)=A_T(X)$ (respectively, $S\colon\Z_m^{(\Z)}\to \Z_m^{(\Z)}$ with $A(X)=A_S(X)$).

\item The equality in \eqref{(3)} gives $\deg(S\circ S') \leq \deg(S) + \deg(S')$ for $S,S' \in \mathrm{CA}_f(\Z_m)$, but equality may fail in general. (Actually, $\deg(S\circ S') = \deg(S) + \deg(S')$, provided $\supp^*(S) = \supp^*(S')$ and in particular, when $S'=S$, so  $\deg(S^n) = n\deg(S)$.) We are providing no proofs since we are not going to use these properties. 

\item The finite fps \eqref{ffps} associated to a non-zero $T=T_{f[l,r]\lambda} \in \mathrm{CA}(\Z_m)$ shows that $T$ is a linear combination of powers of the right shift $\sigma$, as  \[T = A_T(\sigma) = \sum_{i=l}^{r} \lambda_i \sigma^{-i}.\]
The same holds for $S\in\mathrm{CA}_f(\Z_m)$.

\item While $\mathrm{End}(\Z_m^{(\Z)})$ is non-commutative and has cardinality $\mathfrak c$, in view of Theorem~\ref{fpscor}(2), $\mathrm{CA}_f(\Z_m)$ is commutative and countable. Moreover, by Theorem~\ref{finitaryCH}, $\mathrm{CA}_f(\Z_m)$ coincides with the centralizer of $\sigma$ in the $\Z_m$-subalgebra of $\mathrm{End}(\Z_m^{(\Z)})$
consisting of continuous endomorphisms.
\end{enumerate}
\end{remark}

\subsection{The primary factorization}\label{primaryfactorization}

Let $T=T_{f[l,r]\lambda}\colon \Z_m^{\Z}\to \Z_m^{\Z}$ be a linear cellular automaton with memory set $M=\{l,\ldots,r\}$ and local rule $f\colon \Z_m^M\to \Z_m$. Let $p$ be a prime dividing $m$ and let $p^k$ be the maximum power of $p$ dividing $m$.  Denote by $\bar\lambda$ the projection 
of $\lambda\in\Z_m^\Z$ on $\Z_{p^k}^\Z$. Since $\Z_{p^k}^{\Z}$ is a fully invariant subgroup of $\Z_m^{\Z}$ (which is a direct summand), the restriction $T':=T\restriction_{\Z_{p^k}^{\Z}}$ is a linear cellular automaton (this follows for example from Theorem~\ref{finitaryCH}), with memory set $M$ and local rule $f'=f\restriction_{\Z_{p^k}^M}\colon \Z_{p^k}^M\to\Z_{p^k}$, that is, $T'=T_{f'[l,r]\bar\lambda}$. Consequently, for $S=T\restriction_{\Z_m^{(\Z)}}$, we get that $S':=S\restriction_{\Z_{p^k}^{(\Z)}}=S_{f'[l,r]\bar\lambda}$.
\begin{itemize}
\item[(1)] The minimal memory set $\supp(\bar\lambda)$ of $T'$ and $S'$ is contained in the minimal memory set $\supp(\lambda)$ of $T$ and $S$, and this containment can be strict, as, for example, $\supp(\bar\lambda)$ is empty when all coefficients of $\lambda$ are divisible by $p^k$ while $m\ne p^k$.
\item[(2)] On the other hand, $\supp^*(\lambda)\subseteq \supp^*(\bar\lambda)$ and also this containment can be strict (see Example~\ref{ex*}).
\end{itemize}

\begin{definition}
Let $T\colon \Z_m^{\Z}\to \Z_m^{\Z}$ be a linear cellular automaton, $S = T\restriction_{\Z_m^{(\Z)}} $ and let $m=p_1^{k_1}\cdots p_h^{k_h}$ be the prime factorization of $m$. For every $i\in\{1,\ldots,h\}$, let $K_i:=\Z_{p_i^{k_i}}^{\Z}$ and $L_i:=\Z_{p_i^{k_i}}^{(\Z)}$, and call
$$ T^{(i)}:=T\restriction_{K_i } \quad \text{and} \quad S^{(i)} = S\restriction_{L_i} $$
the \emph{$i$-eth primary components of} $T$ and $S$, respectively.
\end{definition}

 The proof of the following lemma,  collecting some basic properties of the primary components, is straightforward and follows from the discussion above. 
 For the sake of simplicity we consider only the case of  finitary linear cellular automata, also in the subsequent discussion. 
 
\begin{lemma} 
Let $S=S_{f[l,r]\lambda}\colon \Z_m^{(\Z)}\to \Z_m^{(\Z)}$ be a finitary linear cellular automaton, let $m=p_1^{k_1}\cdots p_h^{k_h}$ be the prime factorization of $m$, and, for every $i\in\{1,\ldots,h\}$, let $S^{(i)}\colon L_i\to L_i$ be the $i$-eth primary component of $S$. Then, for every $i\in\{1,\ldots,h\}$:
\begin{itemize}
\item[(1)] the local rule of $S^{(i)}$ is the restriction $f\restriction_{ \Z_{ p_i^{k_i} }^M }$, but we write $f$ in \eqref{asineqref} for the sake of simplicity;
\item[(2)] for $\bar\lambda^{(i)}$ the projection of $\lambda\in\Z_m^\Z$ on $L_i$, that is, $\bar\lambda^{(i)}=\pi_i^\Z(\lambda)$, 
\begin{equation}\label{asineqref}
S^{(i)}=S_{f[l,r]\bar\lambda^{(i)}};
\end{equation}
\item[(3)] $l,r\in\supp^*(\lambda)$ if and only if, for every $i\in\{1,\ldots,h\}$ one has $l,r\in\supp^*(\bar\lambda^{(i)})$;
\item[(4)]$\bigcap_{i=1}^h \supp^*(\bar\lambda^{(i)}) = \supp^*(\lambda) \subseteq \supp^*(\bar\lambda^{(i)}) \subseteq \supp(\bar\lambda^{(i)}) \subseteq \supp(\lambda)$.
\end{itemize}
\end{lemma}

\smallskip
Now we consider a finitary linear cellular automaton $S\colon \Z_m^{(\Z)}\to \Z_m^{(\Z)}$ and let for every $i\in\{1,\ldots,h\}$, $L_i=\Z_{m_i}^{(\Z)}$ and $S^{(i)}=S\restriction_{L_i}$. There is an isomorphism $\xi\colon\Z_m^{(\Z)}\to L_1\times\ldots\times L_h$ induced by $\iota$, 
witnessing that $S$ and $S^{(1)}\times\ldots\times S^{(h)}$ are conjugated, that is, $S=\xi^{-1}\circ(S^{(1)}\times\ldots\times S^{(h)})\circ \xi$. 
For every $i\in\{1,\ldots,h\}$, let $\bar S^{(i)}= \xi^{-1}\circ (S^{(i)}\times \prod_{j\neq i, j=1}^{h}id_{L_j})\circ \xi$, i.e., $\bar S^{(i)}$ is conjugated to $S^{(i)}\times \prod_{j\neq i, j=1}^{h}id_{L_j}$ by $\xi$:
$$\xymatrix{
\Z_m^{(\Z)}\ar[rr]^{S}\ar[d]_{\xi} & & \Z_m^{(\Z)}\ar[d]^\xi \\
L_1\times\ldots\times L_h \ar[rr]_{S^{(1)}\times\ldots\times S^{(h)}}& & L_1\times\ldots\times L_h} \quad \quad\quad
\xymatrix{\Z_m^{(\Z)}\ar[rr]^{\bar S^{(i)}}\ar[d]_{\xi} & & \Z_m^{(\Z)}\ar[d]^\xi \\
L_1\times\ldots\times L_h\  \ \ar[rr]_{\ \ S^{(i)}\times \prod_{j\neq i, j=1}^{h}id_{L_j}}& & \ \ L_1\times\ldots\times L_h \, .
}$$
Then obviously
\begin{equation}\label{LAAAAAAAAAST:Eq}
\bar S^{(1)}\circ\ldots\circ \bar S^{(h)} =\xi^{-1}\circ( S^{(1)}\times\ldots\times S^{(h)})\circ \xi = S.
\end{equation}

In the sequel we will identify $\Z_m^{\Z}$ with $K_1\times\ldots\times K_h$ and $L$ with  $L_1\times\ldots\times L_h$, so also $T$ with $T^{(1)}\times\ldots\times T^{(h)}$ and $S$ with $S^{(1)}\times\ldots\times S^{(h)}$, and we formalize this as follows. 

\begin{definition}
Let $T\colon \Z_m^{\Z}\to \Z_m^{\Z}$ be a linear cellular automaton, let $S=T\restriction_{\Z_m^{(\Z)}}$ and let $m=p_1^{k_1}\cdots p_h^{k_h}$ be the prime factorization of $m$. Then  $T=T^{(1)}\times\cdots\times T^{(h)}$ is the \emph{primary factorization of} $T$ and $S=S^{(1)}\times\cdots\times S^{(h)}$ is the \emph{primary factorization of} $S$.
\end{definition}

In the following result we identify also $R_m$ and $R_{m_1}\times\ldots\times R_{m_h}$ by means of the ring isomorphism $j$ from \eqref{jeq}.

\begin{proposition}
 Let $S\in\mathrm{CA}_f(\Z_m)$ and let $A(X)=\varrho(S)=A_S(X)\in R_m$.
Then, for every $i\in\{1,\ldots, h\}$, using the isomorphisms $\varrho_i\colon \mathrm{CA}_f(\Z_{m_i})\to R_{m_i}$ 
and $\varrho\colon\mathrm{CA}_f(\Z_m)\to R_m$ from Theorem~\ref{fpscor}(2),
\begin{itemize}
\item[(1)] $\varrho_i(S^{(i)})=A_i(X)\in R_{m_i},$ where $A_i(X)$ is defined in \eqref{Ai};
\item[(2)] $\varrho(\bar S^{(i)})=B_i(X)\in R_{m},$ where $B_i(X)$ is defined in \eqref{Bi}. 
\end{itemize}
\end{proposition}
\begin{proof}
(1) is clear.

(2)  Due to the above identification, $\bar S^{(i)}= S^{(i)}\times \prod_{j\neq i, j=1}^{h}id_{L_j}$ for every $i\in\{1,\ldots,h\}$. Let $c\in L$. So, $c=(c_1,\ldots,c_h)$ by the identification. This corresponds to $P_c(X)=(P_{c_1}(X),\ldots,P_{c_h}(X))$ and, for $A(X)\in R_m$, $A(X)P_c(X)= \big(A_1(X)P_{c_1}(X), \ldots, A_h(X)P_{c_h}(X) \big)$.
Then, for every $i\in\{1,\ldots,h\}$,
$$B_i(X)P_c(X) = \big(P_{c_1}(X), \ldots, P_{c_{i-1}}, A_i(X)P_{c_i}(X),P_{c_{i+1}}(X),\ldots, P_{c_h}(X) \big) = P_{\bar S^{(i)}(c)}(X)=A_{\bar S^{(i)}}(X)P_c(X).$$
Since this is true for every $c\in\Z_m^\Z$, we get that $B_i(X)=A_{\bar S^{(i)}}(X)$.
\end{proof}

\section{The dual of a linear cellular automaton is a finitary linear cellular automaton}\label{like:a:scalar}

In the sequel we consider an alternative description of the one-dimensional linear cellular automata on $\Z_m$. 
To this end, we introduce a set of sharper notations as follows. Let $K:=\Z_m^\Z$ denote the set of configurations, which is a totally disconnected compact abelian group, and let $L:=\Z_m^{(\Z)}$ denote the set of finite configurations, which we consider as a discrete abelian group. So, as usual, we make the identifications $L=\widehat K$ and $K=\widehat L$.

For $l \leq r$ in $\Z$, let
\begin{equation}\label{Pdef}
P_{l,r}:= \prod_{i=l}^r(\Z_m)_i \leq L,
\end{equation}
and let $\pi_{l,r}\colon K \to P_{l,r}$ be the canonical projection; for $n\in\Z$, we simply write $\pi_n :=\pi_{n,n}$. 

In this notation, let $T=T_{f[l,r]\lambda}\colon\Z_m^\Z\to \Z_m^\Z$ be a linear cellular automaton. First of all, it is convenient to extend the function $f\colon P_{l,r}\to\Z_m$ to $K$ by simply setting $\bar f := f \circ \pi_{l,r}\colon K\to\Z_m$. Then \eqref{lCA} becomes, for every $n\in\Z$,
\begin{equation*}
\pi_{n}(T(x)) =  (T(x))_n=f(x_{n+l},\ldots,x_{n+r}) = (f \circ \pi_{n+l,n+r}) (x) = (f \circ \pi_{l,r} \circ \sigma^n)(x) = \bar f(\sigma^n(x))\in \Z_m.
\end{equation*}
The right-hand side of this chain of equations suggests the following alternative description of $T$. 

\smallskip
In Pontryagin duality theory it is standard to consider the bilinear evaluation map 
\begin{equation}\label{speq}
\langle -\mid-\rangle\colon L\times K\to \Z_m\leq\mathbb{T}
\end{equation}
defined by $$\langle\lambda\mid x\rangle = \lambda(x)=\sum_{i\in \Z} \lambda_ix_i,$$ for $\lambda \in L=\widehat K$ and $x\in K$.
The function defined in \eqref{speq} is $\sigma$-invariant (hence, also $\beta$-invariant), in the obvious sense: for every $(a,x)\in L\times K$, $$
\langle \sigma(a)\mid \sigma (x)\rangle=\langle a\mid x\rangle.$$

Using the bilinear evaluation map in \eqref{speq} we show now that $T_\lambda$ coincides with the function $K\to K$ defined in \eqref{Eq:June10}.

\begin{proposition}\label{TTSS} 
Let $T_{\lambda}\colon K\to K$ be a linear cellular automaton and define $E_\lambda\colon K\to K$ by 
\begin{equation}\label{Eq:June10}
E_\lambda(x):=(\langle \sigma^n(\lambda)\mid x\rangle)_{n\in\Z} \ \mbox{ for every } \ x=(x_n)_{n\in\Z}\in K.
\end{equation}
Then $T_{\lambda}=E_\lambda$, and so $S_{\lambda}=E_\lambda\restriction_L$.
\end{proposition}
\begin{proof}
For every $x\in K$ and $n\in\Z$, we have 
\begin{equation*}
E_\lambda(x)=(\langle \sigma^n(\lambda)\mid x\rangle)_{n\in\Z}
=(\langle (\lambda_{i-n})_{n\in \Z} \mid x\rangle)_{n\in\Z}
=\left(\sum_{i\in\Z} \lambda_{i-n}x_i\right)_{n\in\Z}.
\end{equation*}
So, in view of \eqref{lCA} and \eqref{locRule},  we conclude that $T_{\lambda}=E_\lambda$, and then also that $S_{\lambda}=E_\lambda\restriction_L$.
\end{proof}

Proposition~\ref{TTSS} gives another motivation for the brief notation $T_\lambda$ and $S_\lambda$ that we introduced in Remark~\ref{remark:june:8}(1).

\smallskip
In the next section we will use the following useful computation.

\begin{example}\label{image(e_1)}
Fixed $\lambda\in L$, Proposition~\ref{TTSS} yields that, for every $k\in \Z$, 
\begin{equation}\label{Eqe_1}
T_\lambda(e_k) = \left(\sum_{i\in\Z} \lambda_{i-n} (e_k)_i\right)_{n\in\Z} = ( \lambda_{k-n} )_{n\in\Z} 
= \sum_{n\in\Z} \lambda_{k-n} e_n = \sum_{n = k-r}^{k-l} \lambda_{k-n} e_n \in P_{k-r,k-l}. 
\end{equation} 
By induction, one can deduce that, for every $n\in\N$, $T_{\lambda}^n (e_k) \in P_{k-nr,k-nl}.$
\end{example}

Now we are in a position to prove Theorem~\ref{dualCAintro}.

\begin{proof}[\bf Proof of Theorem~\ref{dualCAintro}] 
Fix $\lambda\in L$ and consider the linear cellular automaton $T_\lambda\colon K\to K$, which is a continuous group endomorphism of the compact abelian group $K$. We prove that its dual endomorphism $\widehat {T_\lambda}\colon L\to L$ coincides with $S_{\lambda^\wedge}$ and that $\widehat {S_\lambda}=T_{\lambda^\wedge}$.  

Let $x\in K$ and $a\in L$. Then $a\in L$ is the character of $K$ such that, for every $x\in X$,
\[
a(x) = \langle a \mid x \rangle = \sum_{i\in \Z} a_i x_i.
\]
For every $n \in \Z$, Proposition~\ref{TTSS} gives $( T_\lambda(x) )_n = \sum_{i\in\Z} \lambda_{i-n} x_i$, while $\widehat {T_\lambda}(a)=a \circ T_\lambda$ by definition, so
\begin{equation*}
(\widehat {T_\lambda}(a))(x) = a ( T_\lambda(x) ) = \langle a \mid T_\lambda(x) \rangle=\sum_{n\in\Z}a_n\left(\sum_{i\in\Z} \lambda_{i-n}x_i\right).
\end{equation*}
As both $a \in L$ and $\lambda \in L$, the above sum is finite, so we can rearrange it as follows to obtain
\[(\widehat {T_\lambda}(a))(x) = \sum_{i\in\Z} \left( \sum_{n\in\Z} a_n\lambda_{i-n} \right) x_{i} 
= \left\langle \left( \sum_{n\in\Z} a_n\lambda_{i-n} \right)_{i\in \Z} \mid x \right\rangle. \]
Moreover, for every $a\in L$,
\begin{equation*}
S_{\lambda^\wedge}(a)=(\langle \sigma^i(\lambda^\wedge)\mid a\rangle)_{i\in\Z}=\left(\sum_{n\in\Z} a_n\lambda^\wedge_{n-i}\right)_{i\in\Z}=\left(\sum_{n\in\Z} a_n\lambda_{i-n}\right)_{i\in\Z}.
\end{equation*}
Then $(\widehat {T_\lambda}(a))(x) = \langle S_{\lambda^\wedge}(a) \mid x \rangle = S_{\lambda^\wedge}(a)(x)$.

\smallskip
The equality $\widehat {S_\lambda}=T_{\lambda^\wedge}$ is obtained from $\widehat {T_{\lambda^\wedge}} = S_{\lambda^{\wedge\wedge}}$, as 
$\lambda^{\wedge\wedge}=\lambda$ and $\widehat{\widehat{T_{\lambda^\wedge}}}=T_{\lambda^\wedge}$, by Pontryagin duality.
\end{proof}

The next observation is exploited in the proofs of Corollaries~\ref{injsurjsurjinj} and~\ref{htopFClCA}.

\begin{lemma}\label{wedgeconj}
Fixed $\lambda\in L$, the two finitary linear cellular automata $S_{\lambda^\wedge}$ and $S_{\lambda}$ on $\Z_m$ are conjugated by the group automorphism ${}^\wedge \colon L \to L$, $a\mapsto a^\wedge$, that is, the following diagram commutes.
\[\xymatrix{L \ar[r]^{S_\lambda}\ar[d]_{{}^\wedge} & L \ar[d]^{{}^\wedge}\\
L \ar[r]_{S_{\lambda^\wedge}}&  L}\]
\end{lemma}
\begin{proof}
Let us show that $S_{\lambda^\wedge} (a^\wedge) = S_{\lambda}(a)^\wedge$ for every $a\in L$:
\[S_{\lambda}(a)^\wedge = ( S_{\lambda}(a)_{-n} )_{n\in \Z} = \left( \sum_{i\in\Z} \lambda_{i+n} a_i \right)_{n\in\Z} 
= \left( \sum_{i\in\Z} \lambda_{-i-n}^\wedge a_{-i}^\wedge \right)_{n\in\Z} 
= \left( \sum_{j\in\Z}\lambda_{j-n}^\wedge a_{j}^\wedge \right)_{n\in\Z} = S_{\lambda^\wedge}(a^\wedge).\qedhere\]
\end{proof}

As a consequence of Theorem~\ref{dualCAintro} and the properties of Pontryagin duality, we obtain the following: 

\begin{corollary}\label{injsurjsurjinj}
Let $T_{\lambda}\colon K\to K$ be a linear cellular automaton. Then $T_{\lambda}$ is injective (respectively, surjective) if and only if  $S_{\lambda}$ is surjective (respectively, injective).
\end{corollary}
\begin{proof} 
Theorem~\ref{dualCAintro} yields that $\widehat{T_\lambda}= S_{\lambda^\wedge}$. 
The standard properties of Pontryagin duality imply that $T_\lambda$  is injective (respectively, surjective) if and only if $\widehat{T_\lambda} = S_{\lambda^\wedge}$ is surjective (respectively, injective). By Lemma~\ref{wedgeconj}, $S_{\lambda^\wedge}$ and $S_{\lambda}$ are conjugated, and this implies that $S_{\lambda^\wedge}$ is surjective (respectively, injective) if and only if $S_{\lambda}$ surjective (respectively, injective). 
\end{proof}

Richardson~\cite[Theorem 3]{Rich} (in the non-necessarily linear case) proved that, for any $\lambda\in L$, $T_\lambda$ is surjective if and only if $S_\lambda$ is injective, while the injectivity of $T_\lambda$ implies the surjectivity of $S_\lambda$, but the converse implication need not hold in general. In the linear case the equivalence between the injectivity of $T_\lambda$ and the surjectivity of $S_\lambda$ was obtained in~\cite[Theorem 2]{ION}.  Their proofs are completely different from ours and they do not use Pontryagin duality, so our present proof represents a new (unifying) approach which allows for shorter and simpler proofs.

\section{Injective, surjective and invertible finitary linear cellular automata}\label{sec:InjSurj}

In this section we characterize surjective one-dimensional finitary linear cellular automata as those that are invertible with inverse a one-dimensional finitary linear cellular automaton. We give separately the case of prime power since under this assumption one obtains some extra information that is not available in general (see Theorem~\ref{injbij} and Example~\ref{ex*}).

\begin{proposition}\label{previous}
If $m=p^k$ with $p$ prime and $k\in\N_+$, and $S\colon \Z_m^{(\Z)} \to \Z_m^{(\Z)}$ is a finitary linear cellular automaton, then the following are equivalent: 
\begin{itemize}
   \item[(1)] $S$ is surjective; 
   \item[(2)] $S$ is bijective (i.e., $S$ is a group automorphism);
   \item[(3)] $S$ is bijective and its inverse is again a linear cellular automaton; 
   \item[(4)] $|\supp^*(S)|=1$; 
   \item[(5)] $S^{ p^{k-1}(p-1) }$ is a power of $\sigma$;
   \item[(6)] $S^{ p^{k-1}(p-1) } = \sigma^{ \pm \deg(S) p^{k-1} (p-1) }$. 
\end{itemize}
\end{proposition}
\begin{proof} The implications (3)$\Rightarrow$(2)$\Rightarrow$(1)$\Leftarrow$(5)$\Leftarrow$(6) are clear. 

For the proof of the remaining implications write $A_S(X)= A_S^*(X)+pA_1(X)$ with $A_1(X)\in R_m$; by Theorem~\ref{fpscor} and Lemma~\ref{aboveabove},
\begin{equation}\label{aboveq}
A_{S^{p^{k-1}}}(X) = A_S (X)^{p^{k-1}} = A_S^*(X)^{p^{k-1}},
\end{equation} 
and by Remark~\ref{suppAsuppT}(2),
\begin{equation}\label{aboveqq}
\supp(A_S^*(X)) = \supp^*(A_S(X)) = -\supp^*(S).
\end{equation}

(4)$\Rightarrow$(6)\&(3) 
By hypothesis, let $\supp^*(S) = \{-n\} \subseteq \Z$. By definition, $\deg(S)=|n|$, and by \eqref{aboveqq}, $\supp(A_S^*(X)) = \{n\}$, so $A_S^*(X)=sX^n$ for some $s\in\Z_m$, with $(s,p)=1$. 
By \eqref{aboveq}, $A_{S^{p^{k-1}}}(X) = A_S^*(X)^{p^{k-1}} = s^{p^{k-1}}X^{n p^{k-1}}$, and so
\begin{equation}\label{correct:eq}
A_{S^{ p^{k-1}(p-1) }}(X) = s^{ p^{k-1}(p-1) }X^{n p^{k-1}(p-1)} = X^{n p^{k-1}(p-1)}
\end{equation} 
as $s^{ p^{k-1} (p-1) }=1$ in $\Z_{ p^{k} }$. This proves that $S^{ p^{k-1}(p-1) } = \sigma^{n p^{k-1}(p-1)}$, hence (6) holds.

In particular, \eqref{correct:eq} shows that $A_S (X)^{ p^{k-1}(p-1) } = X^{n p^{k-1}(p-1)}$ is invertible in $R_m$, and so also $A_S(X)$ is invertible in $R_m$. Then $S$ is invertible with $A_{S^{-1}}(X)=A_S(X)^{-1}$, that is, $S^{-1}$ is a finitary linear cellular automaton. This proves (3). 

(1)$\Rightarrow$(4) Let $S=S_{f[l,r]\lambda}$. Since $S$ is surjective if and only if $S^{p^{k-1}} $ is surjective, 
 in view of \eqref{aboveq} and \eqref{aboveqq}, we have $|\supp^*(S)|\geq 1$, and we can assume without loss of generality that $l,r\in\supp^*(S)=\supp^*(\lambda)$. 
 
It remains to prove that if $|\supp^*(S)| > 1$ (i.e., if $l <r$), then $S$ is not surjective. To this end, we show that if $x\in \Z_m^{(\Z)}\setminus\{0\}$, then $|\supp(S(x))|\geq2$; so in particular all $x\in\Z_m^{(\Z)}$ with $|\supp(x)|=1$ are not in the image of $S$.
In fact, fix $x\in \Z_m^{(\Z)}\setminus\{0\}$, and let $n=\min\supp(x)$ and $N=\max\supp(x)$.
Then $(S(x))_{N-l}=\lambda_lx_N$ and $(S(x))_{n-r}=\lambda_rx_n$. So $(\lambda_l,p)=1=(\lambda_r,p)$ entails $N-l,n-r\in\supp(S(x))$. Clearly, $n\leq N$ and $l<r$ imply that $n-r<N-l$, and so $|\supp(S(x))|\geq 2$.
\end{proof}

Concerning the following result, Ito, Osato and Nasu~\cite{ION} give credit to Amoroso and Cooper~\cite{AC} for the proof of the fact that $S\in\mathrm{CA}_f(\Z_m)$ surjective implies injective.

\begin{theorem}\label{injbij}
Let $S\colon \Z_m^{(\Z)}\to \Z_m^{(\Z)}$ be a finitary linear cellular automaton and $S=S^{(1)}\times\cdots\times S^{(h)}$ the primary factorization of $S$. 
Then the following conditions are equivalent: 
\begin{itemize}
   \item[(1)] $S$ is surjective; 
   \item[(2)] $S$ is bijective (i.e., $S$ is a group automorphism); 
   \item[(3)] $S$ is bijective and its inverse is again a finitary linear cellular automaton; 
   \item[(4)] all primary components $S^{(j)}$ satisfy $|\supp^* (S^{(j)})| =1$.
\end{itemize}
\end{theorem}
\begin{proof}  
Since $S$ has the property (j) from the theorem  for $j= 1,2,3$ precisely when each $S^{(i)}$ has (j), it suffices to apply Proposition~\ref{previous}.  
\end{proof}

As the next example shows, the condition in item (4) of Theorem~\ref{injbij} does not imply that $|\supp^* (S)| =1$.

\begin{example}\label{ex*}
Take $S=S_{f[0,1]\lambda}\colon \Z_6^{(\Z)}\to \Z_6^{(\Z)}$ with local rule defined by $f(x_0,x_1) = 3x_0 + 2x_1$ for every $x_0,x_1\in \Z_6$, that is, $\lambda_0=3$ and $\lambda_1=2$. Then $S$ is surjective by Theorem~\ref{injbij} and $\supp^*(\lambda)=\emptyset$; on the other hand, calling $p_1=2$ and $p_2=3$ in order to follow the above notation, $\supp^*(\bar\lambda^{(1)})=\{0\}$ and $\supp^*(\bar\lambda^{(2)})=\{1\}$.
\end{example}

Moreover, the equivalence between item (4) and item (3) of Theorem~\ref{injbij} can be rephrased as follows.

\begin{corollary}\label{Samurai:Thm}
Let $S=S_{\lambda}\colon \Z_m^{(\Z)}\to \Z_m^{(\Z)}$ be a finitary linear cellular automaton. Then $S$ is bijective and its inverse is again a finitary linear cellular automaton if and only if for every prime factor $p$ of $m$, there exists a unique $i\in\supp(\lambda)$ such that $(p,\lambda_i)=1$. 
\end{corollary}

Ito, Osato and Nasu~\cite{ION} already described the linear cellular automata on $\Z_m$ that are injective, in the $d$-dimensional case. So, for $d=1$ we cover \cite[Theorem~2]{ION}, and also~\cite[Theorem~3]{ION}, where they describe the linear cellular automata some power of whose coincides with the shift. Indeed, in view of Corollary~\ref{injsurjsurjinj}, Theorem~\ref{injbij} gives the following.

\begin{theorem}\label{injbijT}
Let $T\colon \Z_m^\Z\to \Z_m^\Z$ be a linear cellular automaton and $T=T^{(1)}\times\ldots\times T^{(h)}$ be the primary factorization of $T$. 
Then the following conditions are equivalent: 
\begin{itemize}
   \item[(1)] $T$ is injective; 
   \item[(2)] $T$ is bijective (i.e., $T$ is a group automorphism); 
   \item[(3)] $T$ is bijective and its inverse is again a linear cellular automaton; 
   \item[(4)]    all primary components $T^{(j)}$ satisfy $|\supp^* (T^{(j)})| =1$.
\end{itemize}
\end{theorem}

Moreover,~\cite[Theorem~1]{ION} describes the $d$-dimensional linear cellular automata on $\Z_m$ that are surjective. In dimension one this gives:  

\begin{theorem}\label{ION2}
Let $T=T_{\lambda}\colon \Z_m^\Z\to \Z_m^\Z$ be a linear cellular automaton. Then the following conditions are equivalent: 
\begin{itemize}
   \item[(1)] $T$ is surjective; 
   \item[(2)] for every prime factor $p$ of $m$, there exists $i\in\supp(\lambda)$ such that $(p,\lambda_i)=1$;
   \item[(3)] the finitary linear cellular automaton $S_{\lambda}\colon \Z_m^{(\Z)}\to \Z_m^{(\Z)}$ is injective. 
\end{itemize}
\end{theorem}
\begin{proof} 
The equivalence of (1) and (2) is the one-dimensional case of~\cite[Theorem~1]{ION}, while the equivalence of (1) and (3) is Corollary~\ref{injsurjsurjinj}. 
\end{proof}

\begin{example}\label{Exa:InjSurj}
Due to the above theorems concerning injectivity and surjectivity of (finitary) linear cellular automata on $\Z_m$,
one can immediately find examples of surjective non-injective linear cellular automata and of injective non-surjective finitary linear cellular automata. Indeed, it suffices to find a linear cellular automaton $T=T_{f[l,r]\lambda}\colon \Z_m^\Z\to \Z_m^\Z$ such that:
\begin{itemize}
\item[(1)] for every prime factor $p$ of $m$ there exists $i\in\{l,\ldots,r\}$ such that $(p,\lambda_i)=1$;
\item[(2)] there is a prime factor $p$ of $m$ such that $(p,\lambda_i)=1=(p,\lambda_j)$ for at least two indices $i\neq j$ in $\{l,\ldots,r\}$.
\end{itemize}
For example,~\cite[Example 5.2.1]{CC} works perfectly. There, one has $m=3$, $l=-1$, $r=1$ and $f\colon \Z_3^{\{-1,0,1\}}\to \Z_3$ defined by $f(x_{-1},x_0,x_1)=x_{-1}+x_0+x_1$, that is, $\lambda\in\Z_3^\Z$ has $\supp(\lambda)=\supp^*(\lambda)=\{-1,0,1\}$ and $\lambda_{-1}=\lambda_0=\lambda_1=1$. Then this $T_{f[-1,1]\lambda}$ is surjective but non-injective, and equivalently, $S_{f[-1,1]\lambda}$ is injective but not surjective.
\end{example}

\section{The algebraic entropy of finitary linear cellular automata}\label{halgS}

In \S\ref{Sec:Feb26} we recall the necessary background on the algebraic entropy. Then in \S\ref{SSec1} we compute the algebraic entropy of a finitary linear cellular automaton that is leftmost and/or rightmost permutive. The general case is contained in \S\ref{Sec2}.

\subsection{Algebraic entropy}\label{Sec:Feb26}

Following~\cite{DGB0}, let $L$ be an abelian group and $\phi\colon L\to L$ an endomorphism. Denote by $\Pf(L)$ the family of all non-empty finite subsets of $L$.
For $F\in\Pf(L)$ and $n\in\N_+$, the \emph{$n$-th $\phi$-trajectory of $F$} is
$$\T_n(\phi,F):=F+\phi(F)+\ldots+\phi^{n-1}(F).$$
The limit
\begin{equation*}
H_{alg}(\phi,F):={\lim_{n\to\infty}\frac{\log|\T_n(\phi,F)|}{n}}
\end{equation*}
exists and it is the \emph{algebraic entropy of $\phi$ with respect to $F$}.
The \emph{algebraic entropy of $\phi$} is
\begin{equation*}
h_{alg}(\phi):=\sup\{H_{alg}(\phi,F): F\in\Pf(L)\}.
\end{equation*}

\begin{remark}\label{cofinal} In case $E\subseteq F$ are non-empty finite subsets of $L$, $H_{alg}(\phi,E)\leq H_{alg}(\phi,F)$.
So, if $\mathcal F$ is a cofinal subfamily of $(\Pf(L),\subseteq)$, then $h_{alg}(\phi)=\sup\{H_{alg}(\phi,F):F\in \mathcal F\}$.
\end{remark}

For example, we can always assume that $0\in F$, that is, letting $\Pf^0(L):=\{F\in\Pf(L):0\in F\}$, we have that $h_{alg}(\phi)=\sup\{H_{alg}(\phi,F):F\in\Pf^0(L)\}$.

Moreover, in case $L$ is torsion (as for example $L=\Z_m^{(\Z)}$), the family $\mathcal F(L)$ of all finite subgroups of $L$ is cofinal in $\Pf(L)$ as every non-empty finite subset $F$ of $L$ is contained in the finite subgroup $\langle F\rangle$ generated by $F$; so, $h_{alg}(\phi)=\sup\{H_{alg}(\phi,F):F\in\mathcal F(L)\}$.

\smallskip
We recall here the basic properties of the algebraic entropy that we use in this section.

\begin{proposition}\label{properties}
Let $L$ be an abelian group and $\phi\colon L\to L$ an endomorphism.
\begin{itemize}
\item[(1)] \emph{[Invariance under conjugation]} If $M$ is another abelian group, $\eta\colon M\to M$ is an endomorphism and $\phi$ and $\eta$ are conjugated by the isomorphism $\xi\colon L \to M$ (i.e., $\eta=\xi\circ\phi\circ\xi^{-1}$), then $h_{alg}(\phi)=h_{alg}(\eta)$.
\item[(2)] \emph{[Logarithmic Law]} For every $k\in\N_+$, $h_{alg}(\phi^k) = k h_{alg}(\phi)$. If $\phi$ is an automorphism, then $h_{alg}(\phi^k) = |k|h_{alg}(\phi)$ for every $k\in\Z$.
\item[(3)] \emph{[Continuity]} If $L$ is a direct limit of $\phi$-invariant subgroups $\{L_i:i\in I\}$, then $h_{alg}(\phi)=\sup_{i\in I}h_{alg}(\phi\restriction_{L_i})$.
\item[(4)] \emph{[Weak Addition Theorem]} If $L=L_1\times L_2$, $\phi=\phi_1\times \phi_2$ with $\phi_i\colon L_i\to L_i$ and $F_i\in\P_{fin}(L_i)$, $i=1,2$, then $H_{alg}(\phi,F_1\times F_2)=H_{alg}(\phi_1,F_1)+H_{alg}(\phi_2,F_2)$; consequently, $h_{alg}(\phi_1\times\phi_2)=h_{alg}(\phi_1)+h_{alg}(\phi_2)$.
\end{itemize}
\end{proposition}

 Even if we do not apply it in the sequel, next we recall one of the fundamental properties of the algebraic entropy, that is, the so-called Addition Theorem, which in particular covers Proposition~\ref{properties}(4).

\begin{theorem}\emph{\cite{DGB0}}
Let $L$ be a an abelian group, $\phi\colon L\to L$ an endomorphism and $M$ a subgroup of $L$ such that $\phi(M)\subseteq M$. Then $$h_{alg}(\phi)=h_{alg}(\phi\restriction_M)+h_{alg}(\overline\phi_{L/M}),$$ where $\phi\restriction_M\colon M\to M$ and $\overline\phi_{L/M}$ denotes the endomorphism of $L/M$ induced by $\phi$.
\end{theorem}

Also a Uniqueness Theorem is available for the algebraic entropy. For more details also on the connection of the algebraic entropy with number theory, see~\cite{DGB0}.

\smallskip
The following is the so-called Bridge Theorem (see~\cite{DGB1,DGB2}). 
 
\begin{theorem}[Bridge Theorem]\label{BT}
If $K$ is a compact abelian group and $\phi\colon K\to K$ is a continuous endomorphism, then $h_{top}(\phi)=h_{alg}(\widehat\phi)$.
Analogously, if $L$ is a discrete abelian group and $\psi\colon L\to L$ is an endomorphism, then $h_{alg}(\psi)=h_{top}(\widehat\psi)$.
\end{theorem}

The Bridge Theorem connects the algebraic entropy and the topological entropy via Pontryagin duality. On the other hand, it is well known that  the topological entropy of a continuous  surjective endomorphism of a compact group coincides with its measure entropy (proved by Aoki for compact metrizable abelian groups and by Stoyanov~\cite{St} for arbitrary compact groups).

\subsection{Permutive finitary linear cellular automata}\label{SSec1}

For a non-zero finitary linear cellular automaton $S_{f[l,r]\lambda}\colon \Z_m^{(\Z)}\to \Z_m^{(\Z)}$, here we impose the restraint on $\lambda$ to satisfy either $(\lambda_{l},m) = 1$ (i.e., $l\in\supp^*(\lambda)$) or $(\lambda_{r},m)=1$ (i.e., $r\in\supp^*(\lambda)$). These conditions mean respectively that $\lambda_l^n$ and $\lambda_r^n$  are invertible in $\Z_m$ for every $n\in\Z$, and they can be expressed by using the following classical notions, that we introduce in our case of finitary linear cellular automata.

\begin{definition}{\rm\cite{Favati1997,Sab}}\label{left-permutive}
A finitary linear cellular automaton $S_{f[l,r]\lambda}\colon \Z_m^{(\Z)}\to \Z_m^{(\Z)}$ is called:
\begin{itemize}
    \item[(1)] \emph{leftmost} \emph{permutive} if $l < 0$ and $(m, \lambda_l) = 1$ (i.e., $l < 0$ and $l\in\supp^*(\lambda)$);
    \item[(2)] \emph{rightmost} \emph{permutive} if $r > 0$ and $(m, \lambda_r) = 1$ (i.e., $r > 0$ and $r\in\supp^*(\lambda)$);
    \item[(3)] \emph{bipermutive} if it is both leftmost and rightmost permutive.
\end{itemize}
\end{definition}

Now we characterize bipermutivity by means of the degree. The proof is straightforward.

\begin{lemma}\label{permutivedeg}
Let $S=S_{f[l,r]\lambda}\colon \Z_m^{(\Z)}\to \Z_m^{(\Z)}$ be a finitary linear cellular automaton with $l,r\in\supp(\lambda)$.
Then:
\begin{itemize} 
\item[(1)] in case $l\leq r\leq 0$, $l\in\supp^*(\lambda)$ if and only if $\deg(S)=-l$;
\item[(2)] in case $0\leq l\leq r$, $r\in\supp^*(\lambda)$ if and only if $\deg(S)=r$;
\item[(3)] $S$ is bipermutive if and only if $l<0<r$ and $\deg(S)=r-l$.
\end{itemize}
\end{lemma}

\begin{theorem}\label{Th4}
Let $S=S_{f[l,r]\lambda}\colon \Z_m^{(\Z)}\to \Z_m^{(\Z)}$ be a non-zero finitary linear cellular automaton with $l,r\in\supp(\lambda)$.
\begin{itemize}
\item[(1)] If $l\leq r\leq0$ and $l\in\supp^*(\lambda)$, then $h_{alg}(S)=-l\log m$.
\item[(2)] If $0\leq l\leq r$ and $r\in\supp^*(\lambda)$, then $h_{alg}(S)=r\log m$.
\item[(3)] If $l<0<r$ and $l,r\in\supp^*(\lambda)$ (i.e., $S$ is bipermutive), then $h_{alg}(S)= (r-l)\log m$.
\end{itemize}
In all cases, under the corresponding assumption, the conclusion is that $h_{alg}(S)=\deg(S)\log m$.
\end{theorem}
\begin{proof} 
First of all, let us note that the last assertion follows from Lemma~\ref{permutivedeg}.

We recall the following notation introduced in \eqref{Pdef} that we use for the whole proof: for $s,t\in\Z$ with $s\leq t$, $$P_{s,t}=\langle e_{s},\ldots, e_t\rangle.$$

\noindent\textsc{Case $l\leq r\leq0$.} Let $\ell=|l|=-l$ and take $F=P_{-h,k+\ell}$ for fixed $h\in\N$ and $k\in\N_+$.

We prove by induction that, for every $n\in\N_+$,
\begin{equation}\label{Eqnnnm}
\T_n(S,F)= P_{-h,k + n\ell}.
\end{equation}
As a consequence, we get that $|\T_n(S,F)|=m^{h+k+n\ell +1}$, and hence
\begin{equation*}\label{aim1}
H_{alg}(S,F)=\lim_{n\to\infty}\frac{(h+k+n\ell +1)\log m}{n}=\ell\log m.
\end{equation*}
To conclude, in view of Remark~\ref{cofinal} and the fact that every finite subset of $\Z_m^{(\Z)}$ is contained in $P_{-h,\ell+k}$ for some $h\in\N$, $k\in\N_+$,
$$h_{alg}(S)=\sup\{H_{alg}(S,P_{-h,k+\ell}): h\in\N, k\in\N_+\}=\ell\log m.$$

The case $n=1$ in \eqref{Eqnnnm} is clear, as $\T_1(S,F)=F$. Now fix $n\in\N_+$ and assume that \eqref{Eqnnnm} holds.
Clearly, $\T_{n+1}(S,F) = \T_n(S,F) + S^n(F)$ and
\begin{equation*}
P_{-h,k+(n+1)\ell}=P_{-h,k+n\ell+\ell} = P_{-h,k+n\ell} + \langle e_{k+n\ell+1},\ldots, e_{k+n\ell+\ell}\rangle.
\end{equation*}
Therefore, from the inductive hypothesis \eqref{Eqnnnm} and the fact that $S^n(F)\subseteq P_{-h,k+n\ell+\ell}$, we deduce that 
\[
\T_{n+1}(S,F) \subseteq P_{-h,k + n\ell} + P_{-h,k+n\ell+\ell} = P_{-h,k+n\ell+\ell}=P_{-h,k+(n+1)\ell}.
\]
To check that the opposite inclusion holds true, it is enough to take into account that, due to \eqref{Eqe_1},
$$
{S^n(e_{k+1}) \in  \lambda_l^n e_{k+n\ell+1}+P_{-h,k+n\ell},}
$$
so, by the inductive hypothesis \eqref{Eqnnnm},
\[
{\lambda_l^n}e_{k+n\ell+1} \in S^n(e_{k+1}) + P_{-h,k+n\ell} = S^n(e_{k+1})+\T_n(S,F) \subseteq \T_{n+1}(S,F);
\]
since $\lambda_l^n$ is invertible in $\Z_m$, we get that
$$e_{k+n\ell+1}\in \T_{n+1}(S,F)\quad\text{and so that}\quad P_{-h,k+n\ell+1}\subseteq \T_{n+1}(S,F).$$
Analogously, by induction on $i\in\{1,\ldots,\ell\}$, one can prove that, for every $i\in\{1,\ldots,\ell\}$,
$$S^n(e_{k+i}) \in {\lambda_l^n}e_{k+n\ell+i} + P_{-h,k+n\ell}+\langle e_{k+n\ell+1},\ldots,e_{k+n\ell+(i-1)}\rangle={\lambda_l^n e_{k+n\ell+i} +P_{-h,k+n\ell+(i-1)}},$$
so
$$\lambda_l^n e_{k+n\ell+i} \in S^n(e_{k+i})  + P_{-h,k+n\ell+(i-1)} \subseteq \T_{n+1}(S,F);$$
exploiting the invertibility of $\lambda_l^n$  in $\Z_m$, we get that $P_{-h,k+n\ell+i}\subseteq \T_{n+1}(S,F).$
With $i=\ell$, this means that $P_{-h,k+n\ell+\ell}\subseteq \T_{n+1}(S,F)$, and hence the equality $\T_{n+1}(S,F)=P_{-h,k+(n+1)\ell}$. This ends up the inductive proof of \eqref{Eqnnnm}.

\medskip
\noindent \textsc{Case $0\leq l\leq r$.} {Proceed as in the above case, by using $F$ in the family $\{P_{-r-k,h}:h\in \N, k\in\N_+\}$, cofinal in $\P_{fin}(\Z_m^{(\Z)})$.}

\medskip
\noindent \textsc{Case $l<0< r$.} Let  $\ell=|l|=-l$ and $F=P_{-k-r,k+\ell}$ for a fixed $k\in\N$. We prove by induction that, for every $n\in\N_+$,
\begin{equation}\label{Eqnnnn}
\T_n(S,F)= P_{-k-nr,k+n\ell}.
\end{equation}
Then we get that $|\T_n(S,F)|=m^{2k+1+n(r+\ell)}$, and we can conclude that
$$H_{alg}(S,P_{-k-r,k+\ell})=\lim_{n\to\infty}\frac{(2k+1+n(r+\ell))\log m}{n}=(r+\ell)\log m.$$
In view of Remark~\ref{cofinal} and the fact that every finite subset of $\Z_m^{(\Z)}$ is contained in $P_{-k-r,k+\ell}$ for some $k\in\N_+$,  
$$h_{alg}(S)=\sup\{H_{alg}(S,P_{-k-r,k+\ell}):k\in\N_+\}=(r+\ell)\log m.$$

The case $n=1$ in \eqref{Eqnnnn} is clear, as $\T_1(S,F)=F$. Now fix $n\in\N_+$ and assume that \eqref{Eqnnnn} holds.
Clearly, $\T_{n+1}(S,F) = \T_n(S,F) + S^n(F)$ and
\begin{equation}\label{Rrnn}
P_{-k-(n+1)r,k+(n+1)\ell}= P_{-k-nr,k+n\ell}+\langle e_{k+n\ell+1},\ldots, e_{k+n\ell+\ell}\rangle+\langle e_{-k-nr-1},\ldots,e_{-k-nr-r}\rangle.
\end{equation}
Therefore, from the inductive hypothesis \eqref{Eqnnnn} and the inclusion $S^n(F)\subseteq P_{-k-nr-r,k+n\ell+\ell}$, 
we deduce that $\T_{n+1}(S,F) \subseteq P_{-k-nr-r,k+n\ell+\ell}=P_{-k-(n+1)r,k+(n+1)\ell}.$

\smallskip
To check that the opposite inclusion holds true, it is enough to take into account first that, due to \eqref{Eqe_1},
$$S^n(e_k) \in {\lambda_l^n}e_{k+\ell n+1} + P_{-k-nr,k+n\ell},$$
so, by the inductive hypothesis \eqref{Eqnnnn},
$${\lambda_l^n}e_{k+n\ell+1} \in S^n(e_k)  + P_{-k-nr,k+n\ell} = S^n(e_k)  +  \T_n(S,F) \subseteq \T_{n+1}(S,F);$$
since $\lambda_l^n$ is invertible in $\Z_m$, we get that 
\[e_{k+n\ell+1}\in \T_{n+1}(S,F)\quad\text{and so that}\quad P_{-k-nr,k+n\ell+1}\subseteq \T_{n+1}(S,F).\]
Analogously, by induction on $i\in\{1,\ldots,\ell\}$, one can prove that, for every $i\in\{1,\ldots,\ell\}$,
$$S^n(e_{k+i}) \in {\lambda_l^n}e_{k+n\ell+i} + P_{-k-nr,k+n\ell}+\langle e_{k+n\ell+1},\ldots,e_{k+n\ell+(i-1)}\rangle={P_{-k-nr,k+n\ell+(i-1)}},$$
so
$$\lambda_l^n e_{k+n\ell+i} \in S^n(e_{k+i})  + P_{-k-nr,k+n\ell+(i-1)} \subseteq \T_{n+1}(S,F);$$
exploiting the invertibility of $\lambda_l^n$  in $\Z_m$, we get that $P_{-k-nr,k+n\ell+i}\subseteq \T_{n+1}(S,F).$
With $i=\ell$, this means that
\begin{equation}\label{ell}
P_{-k-nr,k+n\ell+\ell}\subseteq \T_{n+1}(S,F).
\end{equation}

Now take into account that, due to \eqref{Eqe_1}, 
$$S^n(e_{-k}) \in {\lambda_r^n}e_{-k-nr-1} + P_{-k-nr,k+n\ell},$$  so, by the inductive hypothesis \eqref{Eqnnnn},
$$\lambda_r^ne_{-k-nr-1} \in S^n(e_{-k})  + P_{-k-nr,k+n\ell} = S^n(e_{-k})  +  \T_n(S,F) \subseteq \T_{n+1}(S,F);$$
{since $\lambda_r^n$ is invertible in $\Z_m$, we get that
$$e_{-k-nr-1}\in \T_{n+1}(S,F)\quad\text{and so that}\quad P_{-k-nr-1,k+n\ell}\subseteq \T_{n+1}(S,F).$$}
Analogously, by induction on $i\in\{1,\ldots,r\}$ and using  invertibility of $\lambda_l^n$  in $\Z_m$, we deduce as above that
\begin{equation}\label{err}
P_{-k-nr-n,k+n\ell}\subseteq \T_{n+1}(S,F).
\end{equation}

Now \eqref{Rrnn}, \eqref{ell} and \eqref{err} entail $P_{-k-(n+1)r,k+(n+1)\ell}\subseteq \T_{n+1}(S,F)$. This establishes the equality $\T_{n+1}(S,F)=P_{-k-(n+1)r,k+(n+1)\ell}$ and ends up the inductive proof of \eqref{Eqnnnn}.
\end{proof}

We immediately obtain the following corollary for $\Z_p$ with $p$ prime, as in this case $\supp(\lambda)=\supp^*(\lambda)$.

\begin{corollary}
If $p$ is a prime, then $h_{alg}(S)=\deg(S)\log p$ for every $S\in \mathrm{CA}_f(\Z_p)$.
\end{corollary}

\subsection{General case}\label{Sec2}

Modulo the Bridge Theorem~\ref{BT}, the following result is an equivalent statement of~\cite[Lemma~1]{DMM}. Let us recall (see Corollary~\ref{CAalg}) that each positive power $S_\lambda^{i}$ of the  finitary linear cellular automaton $S_{\lambda}\colon\Z_m^{(\Z)}\to\Z_m^{(\Z)}$ is a finitary linear cellular automaton itself, and so there is $\lambda'\in \Z_m^{(\Z)}$ such that $S_{\lambda}^{i} = S_{\lambda'}$.

\begin{proposition}\label{above}
Let $m=p^k$ with $p$ prime and $k\in\N_+$ and let $S_\lambda\colon \Z_m^{(\Z)}\to \Z_m^{(\Z)}$ be a finitary linear cellular automaton.
Assume that $\supp^*(\lambda)$ is non-empty and let $l':=\min \supp^*(\lambda)$ and $r':=\max \supp^*(\lambda)$. 
Then ${S_{\lambda'} }:=S_\lambda^{p^{k-1}}$ satisfies $\min\supp(\lambda') = p^{k-1}l'$ and $\max\supp(\lambda') = p^{k-1}r'$. In particular: 
\begin{itemize}
\item[(1)] if $l' < 0$, then $S_{\lambda'}$ is leftmost permutive; 
\item[(2)] if $r' > 0$, then $S_{\lambda'}$ is rightmost permutive;
\item[(3)] if $l' < 0 < r'$, then $S_{\lambda'}$ is bipermutive. 
\end{itemize}
\end{proposition}
\begin{proof}
Let $A(X)=A_{S_\lambda}(X)$. Then $A(X)^{p^{k-1}}=A_{S_{\lambda'}}(X)$, by Theorem~\ref{fpscor}. 
In particular, $\supp( A(X)^{p^{k-1}} ) = -\supp(\lambda')$ by Remark~\ref{suppAsuppT}(2), and then 
\begin{equation}\label{eq:04.02}
\min\supp(\lambda') = - \max\supp( A(X)^{p^{k-1}} )\quad \text{and}\quad \max\supp(\lambda') = - \min\supp(A(X)^{p^{k-1}}).
\end{equation}

Write $A(X)= A^*(X)+pA_1(X)$ for some $A_1(X)\in R_m$. The definition of $l'$ and the equality 
\begin{equation}\label{A2}
A^*(X) = A(X) - pA_1(X)= \lambda_{l'} X^{-l'} + \ldots
\end{equation}
show that the leading monomial in $A^*(X)$ is $\lambda_{l'} X^{-l'}$. 
From Lemma~\ref{aboveabove} and \eqref{A2}, we get that 
\[A (X)^{p^{k-1}} = A^*(X)^{p^{k-1}} =  (\lambda_{l'})^{p^{k-1}} X^{-p^{k-1}l'} + \ldots,\]
and so the leading monomial of $A(X)^{p^{k-1}}$ is $(\lambda_{l'})^{p^{k-1}} X^{-p^{k-1}l'}$ with $(\lambda_{l'})^{p^{k-1}}\neq0$, as 
$(\lambda_{l'}, p ) =1$. Then \eqref{eq:04.02} gives 
\[p^{k-1}l' = - \max\supp( A(X)^{p^{k-1}} ) = \min\supp(\lambda').\]

The same argument shows also that $\max\supp(\lambda') = p^{k-1}r'$. The last part follows from the definitions.
\end{proof}

\begin{proof}[\bf Proof of Theorem~\ref{Th5}] Let $m=p^k$ with $p$ prime and $k\in\N_+$, and let $S_\lambda\colon \Z_m^{(\Z)}\to \Z_m^{(\Z)}$ be a finitary linear cellular automaton.  We have to prove the following.
\begin{itemize}
\item[(1)] If $\supp^*(\lambda)=\emptyset$ (i.e., $p$ divides all coefficients $\lambda_i$), then $S_{\lambda}^k\equiv 0$, and so $h_{alg}(S_\lambda)=0$.
\item[(2)] If $\supp^*(\lambda)\ne \emptyset$, with $l':=\min \supp^*(\lambda)$ and $r':=\max \supp^*(\lambda)$, then 
\begin{equation}\label{Eq:Laaast}
h_{alg}(S_\lambda)=\begin{cases} -l'\log m & \text{if}\ l'\leq r'\leq0,\\
 r'\log m & \text{if}\ 0\leq l'\leq r',\\
 (r'-l')\log m & \text{if}\ l'<0<r'.
\end{cases}
\end{equation}
\end{itemize}

(1) By hypothesis, $A_S^*(X)\equiv 0$, and so there exists $A_1(X)\in R_m$ such that $A_S(X)=pA_1(X)$. By Theorem~\ref{fpscor}, $A_{S^k}(X)=A_S(X)^k=p^kA_1(X)^k\equiv 0$. Therefore, $S_\lambda^{{k}} \equiv 0$ and then $h_{alg}(S_\lambda^{{k}}) = 0$. By the Logarithmic Law (Proposition~\ref{properties}(2)), $h_{alg}(S_\lambda^{{k}}) = {k} h_{alg}(S_\lambda)$, hence $h_{alg}(S_\lambda) = 0$.

(2) Consider the case $l'\leq r'\leq0$. Then, by Proposition~\ref{above}, with $S_{\lambda'}:=S_\lambda^{p^{k-1}}$ we have $p^{k-1}l'=\min\supp (\lambda')$.  
Therefore, Theorem~\ref{Th4} gives that 
$$h_{alg}(S_{\lambda'}) = -p^{k-1}l' \log m,$$
and by the Logarithmic Law (Proposition~\ref{properties}(2)),
$$h_{alg}(S_{\lambda'})= p^{k-1} h_{alg}(S_\lambda);$$ 
hence, $h_{alg}( S_{\lambda} ) = -l' \log m$.

The other two cases can be treated analogously.
\end{proof}

\begin{remark}\label{shift*}
In the above theorem $h_{alg}(S_\lambda)$ depends only on $\supp^*(\lambda)$, not on $\supp(\lambda)$; moreover, the specific values of the coefficients $\lambda_i$, for $i\in \supp^*(\lambda)$, are irrelevant, so they can be taken to be $1$. Hence, as far as only the value of $h_{alg}(S_\lambda)$ is concerned, 
we can assume that $\supp(\lambda)=\supp^*(\lambda)$ and that $\lambda_i\in\{0,1\}$ for every $i\in\Z$.
\end{remark}

Next we prove our most general result, that computes the algebraic entropy of any finitary linear cellular automaton on $\Z_m$. 

\begin{proof}[\bf Proof of Theorem~\ref{halgFClCA}] 
Let $S^{(1)}\times\ldots\times S^{(h)}$ be the primary factorization of $S$. 
The Invariance under conjugation (Proposition~\ref{properties}(1)) and the weak Addition Theorem (Proposition~\ref{properties}(4)) give
\[
h_{alg}(S)=h_{alg}(S^{(1)}\times\ldots\times S^{(h)})=\sum_{i=1}^h h_{alg}(S^{(i)}).
\]
To conclude, for every $i\in\{1,\ldots,h\}$, apply Theorem~\ref{Th5} to $S^{(i)}$ to obtain that $h_{alg}(S^{(i)})=\deg(S^{(i)})\log p_i^{k_i}$.
\end{proof}

As a direct consequence of Theorem~\ref{dualCAintro}, we find the value of the topological entropy of linear cellular automata, that is, \cite[Theorem 2]{DMM}. 

\begin{corollary} \label{htopFClCA} 
Let $T=T_\lambda\colon \Z_m^{\Z}\to \Z_m^{\Z}$ be a linear cellular automaton with primary factorization $T=T^{(1)}\times\ldots\times T^{(h)}$. 
Then $h_{top}(T)=\sum_{i=1}^h \deg(T^{(i)})\log p_i^{k_i}$.
\end{corollary}
\begin{proof}
By the Bridge Theorem~\ref{BT} and Theorem~\ref{dualCAintro}, $h_{top}(T_\lambda)=h_{alg}(\widehat{T_\lambda})=h_{alg}(S_{\lambda^\wedge})$. By Lemma~\ref{wedgeconj}, $S_{\lambda^\wedge}$ and $S_\lambda$ are conjugated, so $h_{alg}(S_{\lambda^\wedge})=h_{alg}(S_\lambda)$ by the Invariance under conjugation (Proposition~\ref{properties}(1)). As we get that $h_{top}(T_\lambda)=h_{alg}(S_\lambda)$, it suffices to apply Theorem~\ref{halgFClCA}.
\end{proof}

\subsection{Final remarks and open questions}\label{finalsec}

The classical Kaplansky paradigm connects endomorphisms of an abelian group $L$ and $\Z[X]$-module structures of $L$ by assigning to $\phi\in\mathrm{End}(L)$ the $\Z[X]$-module structure of $L$ defined by $X\cdot a:= \phi(a)$, for $a\in L$; in the opposite direction, every polynomial $P(X)\in \Z[X]$ gives rise to $\phi\in\mathrm{End}(L)$ by setting $\phi(a):=P(X)\cdot a$ for $a\in L$. Of course, this paradigm can be specified for the case when $\phi\in \mathrm{Aut}(L)$, in which case one obtains a $\Z[X,X^{-1}]$-module structure on $L$ by putting additionally $X^{-1}\cdot a:= \phi^{-1}(a)$ for every $a\in L$, and extending this action by linearity on all Laurent polynomials $A(X)\in \Z[X,X^{-1}]$. 

In case $L= \Z_m^{(\Z)}$ and $\phi=\sigma\colon L\to L$, the situation considered in this paper is exactly the classical Kaplansky paradigm described above. Moreover, as pointed out in Remark~\ref{New:Remark}(1), $L$ carries a natural $R_m$-module structure, and according to Theorem~\ref{fpscor} a Laurent polynomial $A(X)\in R_m$ acts on $L$ as the finitary linear cellular automaton $S\colon L\to L$ with $A_S(X) = A(X)$, that is, 
$$S\colon a \mapsto A(X)\cdot a\ \text{for every}\ a\in L.$$
In Theorem~\ref{halgFClCA} we computed the algebraic entropy of this endomorphism (see Question \ref{Ques:Kapla} and the discussion there related to this issue).

\subsubsection{Finitary linear cellular automata with the same algebraic entropy of a power of the shift}\label{astheshift}

Since for any finitary linear cellular automaton $S=S_{f[l,r]\lambda}\colon \Z_m^{(\Z)}\to \Z_m^{(\Z)}$,
\[S = \sum_{i=l}^{r} \lambda_i \sigma^{-i} = A(\sigma)\]
by Remark~\ref{New:Remark}(3), it is natural to study the relation between $h_{alg}(S)$ and, with $n\in\Z$,  $$h_{alg}(\sigma^n)=|n|\log m=\log(m^{|n|}).$$

\smallskip
In the case when $m=p^k$ for a prime $p$ and $k\in\N_+$, a closer look at Theorem~\ref{Th5} shows that, to the effect of the values of the algebraic entropy, the shifts are dominant in the following sense. 

\begin{proposition}\label{allshift} 
Let $m=p^k$, where $p$ is a prime and $k\in\N_+$, and let $S=S_{\lambda}\colon \Z_m^{(\Z)}\to\Z_m^{(\Z)}$ be a finitary linear cellular automaton. Then $h_{alg}(S)=h_{alg}(\sigma^{\deg(S)})$.
\end{proposition}
\begin{proof}
In view of Remark~\ref{shift*}, assume that $\supp(\lambda)=\supp^*(\lambda)$ and $\lambda_i\in\{0,1\}$ for every $i\in\Z$. 
Therefore, $S$ satisfies the hypotheses of Theorem~\ref{Th4}, which guarantees that $h_{alg}(S)=h_{alg}(\sigma^{\deg(S)})$.
\end{proof}

\begin{remark}
The same conclusion of the above proposition can be achieved when $S$ is bijective, without making recurse to the full power of Theorem~\ref{Th4}. Indeed, under this additional assumption, $S^{ p^{k-1}(p-1) } = \sigma^{ \pm \deg(S) p^{k-1}(p-1)}$ by Proposition~\ref{previous}. Then the Logarithmic Law (Proposition~\ref{properties}(2)) gives
\[
p^{k-1} (p-1) h_{alg} ( S ) = h_{alg} ( S^{ p^{k-1} (p-1) } ) = h_{alg} ( \sigma^{ \pm \deg(S) p^{k-1} (p-1) } ) = p^{k-1} (p-1) h_{alg} ( \sigma^{ \deg(S) } ),
\]
and so $h_{alg}( S ) = h_{alg} ( \sigma^{ \deg(S) } )$.
\end{remark}

In particular, when $m$ is a prime power, $n\in\N$, and $\lambda_n$ is coprime with $m$, it turns out that 
$$h_{alg}(\lambda_0 id_G + \lambda_1\sigma + \ldots + \lambda_n \sigma^n)=h_{alg}(\lambda_n\sigma^n) =h_{alg}(\sigma^n),$$
 which somehow corresponds to the intuitive understanding that the asymptotic behaviour of a polynomial is the same as that of its leading monomial. It is not clear whether this property of the pair $L= \Z_m^{(\Z)}$ and its automorphism $\sigma\colon L\to L$ can be extended to the general situation described in the beginning of this subsection. Namely:

\begin{question}\label{Ques:Kapla}
If $L$ is a (torsion) abelian group and $\phi\colon L\to L$ an endomorphism, can the algebraic entropy $h_{alg}(P(\phi))$, where $P(X) \in \Z[X]$, be estimated in terms of $h_{alg}(\phi)$ and some invariants of $P(X)$ (e.g., degree, etc.)? 
\end{question}

Some simple known instances of the above situation are the monomials $P(X) = X^n$; now the Logarithmic Law (Proposition~\ref{properties}(2)) gives $h_{alg}(\phi^n) = nh_{alg}(\phi)$.

\smallskip
As an immediate consequence of Theorem~\ref{Th4}, the conclusion of Proposition~\ref{allshift} holds for an arbitrary $m$, when $S\in \mathrm{CA}_f(\Z_m)$ is bipermutive:

\begin{corollary}
Let $S =S_{f[l,r]\lambda}\in \mathrm{CA}_f(\Z_m)$ with $l,r\in\supp(\lambda)$.  If
\begin{itemize}
  \item[(1)]  $l\leq r\leq0$ and $l\in\supp^*(\lambda)$, or
  \item[(2)] $0\leq l\leq r$ and $r\in\supp^*(\lambda)$, or
  \item[(3)] $l<0<r$ and $l,r\in\supp^*(\lambda)$ (i.e., $S$ is bipermutive), 
\end{itemize}
then $h_{alg}(S)=h_{alg}(\sigma^{\deg(S)})$.
\end{corollary}

In the case when $m$ is a prime power, every finitary linear cellular automaton $S$ satisfies $h_{alg}(S)=h_{alg}(\sigma^{\deg(S)})$ by Proposition~\ref{allshift} (and there are non-permutive cellular automata in this case), so the implications in the above corollary cannot be inverted. 

\smallskip
The property in Proposition~\ref{allshift} cannot be extended to the case when $m$ is not a prime power. 
Now we show a finitary linear cellular automaton on $\Z_m$ that has algebraic entropy different from the algebraic entropy of any power of the right shift $\sigma$.

\begin{example}\label{item1} 
Take the finitary linear cellular automaton $S=S_{f[-1,0]\lambda}\colon\Z_6^{(\Z)}\to \Z_6^{(\Z)}$ with $\lambda_{-1}=4$ and $\lambda_0=3$. Then, in the notation of Theorem~\ref{halgFClCA}, $p_1=2$, $p_2=3$, so we have 
$$A_S(X)=3+4X \in R_6,\quad A_{ S^{(1)} }(X) = 1\in R_2,\quad A_{ S^{(2)} }(X) = X \in R_3.$$ 
Then $\deg(S^{(1)})=0$ and $\deg(S^{(2)})=1$. 
Hence, by Theorem~\ref{halgFClCA}, $h_{alg}(S)=\log3\neq \log 6^{|n|} = h_{alg}(\sigma^{|n|})$ for every $n\in\Z$.
\end{example}

It is then natural to ask which $S\in\mathrm{CA}_f(\Z_m)$ have the same algebraic entropy as a power of $\sigma$, that is, those $S$ with $h_{alg}(S)=\log(m^{|n|})$ for some $n\in\Z$: 

\begin{proposition}\label{item2}
Let $S\colon \Z_m^{(\Z)}\to \Z_m^{(\Z)}$ be a finitary linear cellular automaton and $S^{(1)}\times\ldots\times S^{(h)}$ be the primary factorization of $S$. 
For $n\in\Z$, $h_{alg}(S)=h_{alg}(\sigma^n)$ if and only if $\deg(S^{(i)})=|n|$ for every $i\in\{1,\ldots,h\}$.
\end{proposition}
\begin{proof}
Let $m=p_1^{k_1}\cdots p_h^{k_h}$ be the prime factorization of $m$.
For $i\in\{1,\ldots,h\}$, let $m_i=p_i^{k_i}$ and $d_i=\deg(S^{(i)})$. By Theorem~\ref{halgFClCA}, $$h_{alg}(S)=\sum_{i=1}^h d_i\log m_i=\sum_{i=1}^h \log(m_i^{d_i})=\log\left(\prod_{i=1}^h m_i^{d_i}\right),$$ while, for a fixed $n\in\Z$,
$$h_{alg}(\sigma^n)=\log(m^{|n|}).$$
Then $h_{alg}(S)=h_{alg}(\sigma^{|n|})$ precisely when $\prod_{i=1}^h m_i^{d_i}=m^{|n|}$, that is, for every $i\in\{1,\ldots,h\}$, $d_i=|n|$.
\end{proof}

In the next example we see, among others, that $|n|$ in Proposition~\ref{item2} need not be $\deg(S)$.

\begin{example}\label{non-injex}
Take again $m=6$, and consider the finitary linear cellular automaton $S=S_{f[-1,1]\lambda}\colon \Z_6^{(\Z)}\to \Z_6^{(\Z)}$ with $\lambda_{-1}=4$, $\lambda_0=0$, $\lambda_1=3$. In the notation of Theorem~\ref{halgFClCA}, $p_1=2$ and $p_2=3$, and we get 
\[A_S(X)=3X^{-1}+4X \in R_6,\quad A_{ S^{(1)} }(X) = X^{-1}\in R_2,\quad A_{ S^{(2)} }(X) = X \in R_3, \]
and
$$\deg(S)=0,\quad \deg(S^{(1)})=1,\quad \deg(S^{(2)})=1.$$ 
Then:
\begin{itemize}
\item[(1)] $S$ is neither leftmost nor rightmost permutive, but $S$ is bijective by Theorem~\ref{injbij};
\item[(2)] $S$ has the same algebraic entropy of $\sigma_6$ as $h_{alg}(S) = \log2 + \log3 = \log 6 = h_{alg}(\sigma_6)$ by Theorem~\ref{halgFClCA};
\item[(3)] on the other hand, $h_{alg}(\sigma_6^{\deg(S)}) = h_{alg}(\sigma_6^0) =0 \neq h_{alg}(S)$, so $|n|$ in Proposition~\ref{item2} need not be $\deg(S)$;
\item[(4)] $\supp^*(\lambda)$ is empty, so $\deg(S)=0$, but $h_{alg}(S) \neq 0$ (and in particular $S$ is not nilpotent), so this example also shows that Theorem~\ref{Th5} does not hold when $m$ is not a prime power.
\end{itemize}
\end{example}

\begin{remark} 
For $S\in\mathrm{CA}_f(\Z_m)$, having the same algebraic entropy of a power of $\sigma$ is independent on the bijectivity of $S$.
Indeed, in one direction $S$ in Example~\ref{item1} is bijective in view of Theorem~\ref{injbij}, and it has algebraic entropy different from the algebraic entropy of any power of $\sigma$, according to Proposition~\ref{item2}, as $\deg(S^{(1)}) \neq \deg(S^{(2)})$.

On the other hand, it is easy to find non-bijective finitary linear cellular automata that have the same algebraic entropy of a power of the right shift; indeed, let $T\colon \Z_3^{\Z} \to \Z_3^{\Z}$ be the linear cellular automaton from Example~\ref{Exa:InjSurj}, and consider $S = T \restriction_{ \Z_3^{(\Z)} }\colon \Z_3^{(\Z)} \to \Z_3^{(\Z)}$. Then $S$ is not bijective by Proposition~\ref{previous}, while $h_{alg}(S)=h_{alg}(\sigma^{\deg(S)})$ by Proposition~\ref{allshift}.
\end{remark}

Let $S\colon \Z_m^{(\Z)}\to \Z_m^{(\Z)}$ be a finitary linear cellular automaton, $m=p_1^{k_1}\cdots p_h^{k_h}$ be the prime factorization of $m$ and $S^{(1)}\times\ldots\times S^{(h)}$ be the primary factorization of $S$. For every $i\in\{1,\ldots,h\}$, let $m_i=p_i^{k_i}$. As a consequence of Proposition~\ref{item2} we find that even in case every $S^{(i)}$ is a power of $\sigma_{m_i}$, $S$ need not be a power of $\sigma_m$ as it can have algebraic entropy different from $h_{alg}(\sigma_m^n)$ for every $n\in\Z$. Indeed, in Example~\ref{item1} we have $S^{(1)}=id=\sigma_2^{0}$  
and $S^{(2)}=\sigma_3$, while $S$ is not a power of $\sigma_6$, according to Proposition~\ref{item2}, as $\deg(S^{(1)}) \neq \deg(S^{(2)})$.

What one can say in the general case when $m$ is not a power of a prime is the following direct consequence of Proposition~\ref{allshift} and the weak Addition Theorem (Proposition~\ref{properties}(4)).

\begin{corollary}\label{Coro:Last}
Let $S\colon \Z_m^{(\Z)}\to \Z_m^{(\Z)}$ be a finitary linear cellular automaton, $m=p_1^{k_1}\cdots p_h^{k_h}$ be the prime factorization of $m$, and $S^{(1)}\times\ldots\times S^{(h)}$ be the primary factorization of $S$. For every $i\in\{1,\ldots,h\}$, let $m_i=p_i^{k_i}$, and $s_i=\deg(S^{(i)})$. Then
\[h_{alg}(S)=\sum_{i=1}^hh_{alg}\bigl(\sigma_{m_i}^{s_i}\bigr)=h_{alg}\left( \sigma_{m_1}^{s_1} \times \ldots \times \sigma_{m_h}^{s_h} \right).\]
\end{corollary}

\subsubsection{Finitary linear cellular automata conjugated to a power of the shift}\label{lastsubsub}

In \S\ref{astheshift}, we have discussed when a finitary linear cellular automaton $S$ on $\Z_m$ has the same algebraic entropy of an integer power the right shift $\sigma$ of $\Z_m$. In view of the Invariance under conjugation (Proposition~\ref{properties}(1)), in case $S$ and $\sigma^n$, with $n\in\Z$, are conjugated, then $h_{alg}(S)=h_{alg}(\sigma^n)$. So, the following general question arises.

\begin{question}
Which $S\in\mathrm{CA}_f(\Z_m)$ are conjugated to $\sigma^n$ for some $n\in\Z$?
\end{question}

First of all, note that such an $S$ has to be necessarily bijective.

\begin{remark}\label{sigmarem}
For $n,\ell\in\Z$, $\sigma^n$ and $\sigma^\ell$ are conjugated if and only if $|n|=|\ell|$.

In fact, if $\sigma^n$ and $\sigma^m$ are conjugated, then $|n|\log m=h_{alg}(\sigma^n)=h_{alg}(\sigma^\ell)=|\ell|\log m$.
Vice versa, $\sigma^n$ and $\sigma^{-n}$ are conjugated for every $n\in\Z$, as $\sigma$ and $\sigma^{-1}$ are conjugated by the automorphism from Lemma~\ref{wedgeconj}.
\end{remark}

In view of the above observations, one can rephrase the above problem as follows. In fact, Example~\ref{item1} shows a bijective finitary linear cellular automaton that cannot be conjugated to any integer power of the right shift, since its algebraic entropy differs from the algebraic entropy of any integer power of the right shift. 

\begin{question}
Let $S\in\mathrm{CA}_f(\Z_m)$ be bijective. If $h_{alg}(S)=h_{alg}(\sigma^n)$ for some $n \in \Z$, is it true that $S$ is conjugated to $\sigma^{n}$?
\end{question}

In the case when $m$ is a prime power we can state the following question in view of Proposition~\ref{previous}. 

\begin{question}\label{question}
Let $m=p^k$ with $p$ a prime and $k\in\N_+$, and let $S\in\mathrm{CA}_f(\Z_m)$ be bijective. Is $S$ conjugated to the power $\sigma^{\deg(S)}$?
\end{question}
 
The next result gives a necessary condition for $S$ to be conjugated to a power of $\sigma$.
 
\begin{theorem}\label{thelasttheorem}
Let $m=p^k$ with $p$ a prime and $k\in\N_+$, and let $S\in\mathrm{CA}_f(\Z_m)$ be bijective.
If $S$ is conjugated to a power $\sigma^n$ of a shift, then $n = \pm \deg(S) $.
In particular, if $S$ and $\sigma$ are conjugated, then $\deg(S) = 1$. 
\end{theorem}
\begin{proof}
Assume that there exists an automorphism $\xi\colon\Z_m^{(\Z)}\to \Z_m^{(\Z)}$ such that $S= \xi^{-1} \circ \sigma^{n}\circ \xi$.
Proposition~\ref{previous} gives that
\[
\sigma^{ \pm \deg(S) p^{k-1} (p-1) } = S^{ p^{k-1} (p-1) } = \xi^{-1} \circ\sigma^{ n p^{k-1} (p-1) }\circ \xi.
\]
In particular, $\sigma^{ \pm \deg(S) p^{k-1} (p-1) }$ and $\sigma^{ n p^{k-1} (p-1) }$ are conjugated, so $\pm \deg(S) p^{k-1} (p-1) = n p^{k-1} (p-1)$ by Remark~\ref{sigmarem}, and we get that $\deg(S) = |n|$.
%
\end{proof}

When $m$ is a prime, this condition becomes: 

\begin{corollary}
Let $m=p$ be a prime and let $S\in\mathrm{CA}_f(\Z_m)$.
If $S$ and $\sigma$ are conjugated, then $S^{ p-1 } = \sigma^{\pm (p-1) }$.
\end{corollary}
\begin{proof}
If  $S$ and $\sigma$ are conjugated, then $S$ is bijective, so Proposition~\ref{previous} gives that $S^{ p-1 } = \sigma^{ \pm\deg(S)(p-1) }$.
As $S$ and $\sigma$ are conjugated, $\deg(S)=1$ by Theorem~\ref{thelasttheorem}. Therefore, $S^{ p-1 } = \sigma^{\pm (p-1) }$.
%
\end{proof}

In particular, for $S\in\mathrm{CA}_f(\Z_2)$, $S$ and $\sigma$ are conjugated if and only if $S=\sigma^{\pm1}$.

\smallskip
 The above two results inspire the following more general question.

\begin{question}
Which are the automorphisms of $\Z_m^{(\Z)}$ conjugated to the right shift $\sigma$ of $\Z_m$ or its powers $\sigma^n$ with $n\in\Z$?
What about the endomorphisms of $\Z_m^{(\Z)}$ conjugated to a fixed $S\in\mathrm{CA}_f(\Z_m)$ (in particular, do they still belong to $\mathrm{CA}_f(\Z_m))$?
\end{question}

\end{document}